\documentclass[10pt,amssymb]{amsart}
\usepackage{graphicx}
\usepackage{amsmath,amssymb,amsfonts,euscript,enumerate}
\usepackage{amsthm}
\usepackage{color,wrapfig}
\usepackage{pxfonts}
\usepackage{verbatim}
\newtheorem{thm}{Theorem}[section]
\newtheorem{cor}[thm]{Corollary}
\newtheorem{lem}[thm]{Lemma}

\newcommand{\R}{{\mathbb{R}}}

\newcommand{\vp}{\varphi}
\newcommand{\osc}{\operatornamewithlimits{osc}}

\newcommand{\supp}{\operatorname{supp}}
\newcommand{\D}{\nabla}

\newcommand{\La}{\triangle}
\newcommand{\bs}{\backslash}

\begin{document}

\title{Long time existence of smooth solution for the Porous Medium Equation in a bounded domain}

\author[Sunghoon Kim]{Sunghoon Kim}
\address{Sunghoon Kim:
Department of Mathematics and PMI (Pohang Mathematics Institute), Pohang University of Science and Technology (POSTECH),\\
Hyoja-Dong San 31, Nam-gu, Pohang 790-784, South Korea 
}
\email{math.s.kim@postech.ac.kr}

\maketitle

%% \abstractname
\begin{abstract}
In this paper, we are going to show the long time existence of
the smooth solution for the porous medium equations in a smooth
bounded domain:
\begin{equation}
\begin{cases}
u_t=\La u^m\quad\text{in $\Omega\times [0,\infty)$}\\
u(x,0)=u_0>0\quad\text{in $\Omega$}\\
u(x,t)=0\quad\text{for $x\in\partial\Omega$}
\end{cases}
\end{equation}
where $m>1$ is the permeability. The proof is based on  the short time existence of $C^{2,\overline{\gamma}}_{s}$-smooth solution, the global $C^{1}_{s}$-estimate, the H\"older estimate of divergence type degenerate equation with measurable coefficients and $C^{1,\overline{\gamma}}_s$-estimate of mixed type equation with Lipschitz coefficients.
\end{abstract}

%% \end{abstract}

%% \begin{keyword}



%% Porous Medium Equation in a bounded domain, H\"older estimate, Long time existence

%% \PACS

%% \end{keyword}

\setcounter{equation}{0}
\setcounter{section}{0}
\section{introduction}
\setcounter{equation}{0}
\setcounter{thm}{0}

We consider in this paper the initial value problem for the
\emph{Porous Medium Equation}(PME)
\begin{equation}\label{eq-main}
\begin{cases}
u_t=\La u^m\qquad \qquad \qquad \text{in $\Omega$}\\
u(x,0)=u_0>0\quad \qquad \text{in $\Omega$}\\
u(x,t)=0\qquad \qquad \text{for $x\in\partial\Omega$}
\end{cases}
\end{equation}
posed in a bounded domain $\Omega$ with the range of exponents $m>1$, with initial data $u_0$
nonnegative, integrable and compactly supported. \\
\indent In \cite{KL}, Kim and Lee dealt with the short time existence of solution to PME in a bounded domain. More precisely, the main result in their paper is that, if $u$ is a solution to \eqref{eq-main} and $f=u^m$, then, in some regularity conditions on the initial data $f^0=u_0^m$ and its first and second derivatives, the solution of the degenerate equation
\begin{equation}\label{eq-main-for-f}
f_t=mf^{\frac{m-1}{m}}\La f=mf^{\alpha}\La f, \qquad \qquad \left(f=u^m, \quad \alpha=1-\frac{1}{m}\right)
\end{equation}
exists on a short time interval $[0,T)$ and $f\in C_s^{2,\overline{\gamma}}(\Omega)$ on $[0,T)$,  i.e., the solution and its first and second derivative are H\"older continuous with respect to a suitable Riemannian metric $s$ under which the distance between two points $x_1$ and $x_2$ in $\Omega$ is equivalent to the function
\begin{equation*}
\frac{|x_1-x_2|}{d(x_1)^{\alpha}+d(x_2)^{\alpha}+\sum_{i=1}^{n-1}|x_1-x_2|^{\alpha}}
\end{equation*}
with $d=d(x)$ denoting the distance to the boundary of $\Omega$.\\
\indent In this work, we will show that, under the same assumptions, the solution $f$ of \eqref{eq-main-for-f} exists on the long time $[0,\infty)$ and it is also of class $C_s^{2,\overline{\gamma}}$, i.e., assuming that the initial value $f^0$ is strictly positive in the interior of a domain $\Omega\subset\R^n$, with $f^0=0$ on the boundary, and denoting by $d$ the distance to the boundary of $\Omega$, we will obtain the following result.
\begin{thm}\label{The-Main-Theorem-of-this-paper}
If the functions $f^0$, $Df^0$ and $d^{\alpha}D^2f^0$, restricted to the compact domain $\Omega$, extended continuously up to the boundary of $\Omega$, with extensions which are H\"older continuous on $\Omega$ of class $C^{\overline{\gamma}}(\Omega)$, for some $\overline{\gamma}>0$ and $Df_0\neq 0$ along $\partial\Omega$, then the initial value problem
\begin{equation}\label{eq-cases-main-of-this-paper-section-PMT}
\begin{cases}
f_t=\frac{1}{m}f^{\alpha}\La f, \qquad (x,t)\in\Omega\times(0,\infty),\\
f(x,0)=f^0(x), \qquad x\in\Omega\\
f(x,t)=0, \qquad x\in\partial\Omega\times[0,\infty)
\end{cases}
\end{equation}
admits a solution $f$ which is $C_s^{2,\overline{\gamma}}$-smooth up to the boundary, when $0<t<\infty$.
\end{thm}
As in the Section 4 in \cite{KL}, the coordinate change, $\left(z=f(x',x_n,t)\to x_n=h(x',z,t)\right)$, converts domain
\begin{equation*}
\Omega\in\R^n \quad \Rightarrow \quad \mathcal{D}\in \R^{n-1}\times\R^+
\end{equation*}
and the equation \eqref{eq-main-for-f} to 
\begin{equation}\label{eq-for-sol-h-of-fixed-bondary-problem-in-intro}
h_t=z^{\alpha}\Bigg[{\La}_{x'}h+\Bigg(-\frac{1+|\nabla_{x'}h|^2}{h_{z}}\Bigg)_{z}\Bigg], \qquad \left(\alpha=1-\frac{1}{m}\right).
\end{equation} 
In addition, $h_{x_i}$, $(i=1,\cdots,n-1)$ satisfies
\begin{equation}\label{eq-for-sol-h-x-i-of-fixed-bondary-problem-in-intro}
w_t=z^{\alpha}\nabla_{k}\left(a^{kl}\nabla_{l}w\right), \qquad \left(k,l=1,\cdots,n\right)
\end{equation}
where 
\begin{equation*}
a^{nn}=\frac{1+|\nabla_{x'}h|^2}{h_z^2}, \quad a^{k'n}=-\frac{2h_{k'}}{h_n},\quad a^{nk'}=0 \quad \mbox{and}\quad a^{k'l'}=\delta_{kl} ,\qquad (k',l'=1,\cdots,n-1).
\end{equation*}
\indent Since the solution $f$ of \eqref{eq-cases-main-of-this-paper-section-PMT} is strongly related to the solution $h$ of \eqref{eq-for-sol-h-of-fixed-bondary-problem-in-intro}, by the Theorem 1.1 in \cite{KL}, the solution $h$ exists on a short time interval. Let $(0,T)$ be the maximal interval of existence for $C_{s}^{2,\overline{\gamma}}$-solution $h$ and suppose that $T<\infty$. Then, solution does not belong to the space $C_{s}^{2,\overline{\gamma}}$ at time $t=T$ anymore. However, the existence theory gave $h\in W^{1,2}$. This implies that the coefficients of the equation \eqref{eq-for-sol-h-of-fixed-bondary-problem-in-intro} are only measurable and bounded. Thus, if it could be shown that $h\in C_s^{1,\overline{\gamma}}$, then the coefficients in \eqref{eq-for-sol-h-of-fixed-bondary-problem-in-intro} would belong to $C_s^{\overline{\gamma}}$. Then, Schauder estimate in \cite{KL} provides $C^{2,\overline{\gamma}}_s$ regularity on $h$ at $t=T$ and we can get a extended interval $[0,T')$, $(T<T')$ in which $f\in C_{s}^{2,\overline{\gamma}}$. This contradicts the maximality of $T$. Therefore, $T$ must be infinity and the Theorem \ref{The-Main-Theorem-of-this-paper} follows. Hence, the missing step for the regularity problem to be solved is
\begin{equation}\label{eq-missing step for the regularity}
h\in W^{1,2} \Rightarrow h\in C_s^{1,\overline{\gamma}}, \qquad (0<\overline{\gamma}<1).
\end{equation}
\indent In this paper, we show the long time existence of solution $f=u^m$ by solving the missing step \eqref{eq-missing step for the regularity}\\
\indent The paper is divided into five parts: In Part 1 (Section 2) we review the metric $ds$ which controls the diffusion and state theorem for the short time existence of $f=u^m$ in \cite{KL}. In Part 2 (Section 3) we deal with global estimates, including Gradient estimate and Non-degeneracy. In part 3 and 4 (Section 4 and 5) we establish the H\"older estimate for the solution of degenerated parabolic equation with divergence type and $C^{1,\overline{\gamma}}_S$ estimate for the solution of mixed equation with Lipschitz coefficients. Finally, in Section 6, we prove the long time existence of the solution $f$ of \eqref{eq-main-for-f} which is in $C_s^{2,\overline{\gamma}}(\Omega)$.

\section{Preliminaries}
\setcounter{equation}{0}
\setcounter{thm}{0}

The diffusion in \eqref{eq-for-sol-h-of-fixed-bondary-problem-in-intro} is governed by the Riemannian metric $ds$ where
\begin{equation*}
ds^2=\frac{dx_1^2+\cdots+dx_n^2}{2x_n^{\alpha}}.
\end{equation*}
The distance between two points $x^1=(x_1^1,\cdots,x_n^1)$ and $x^2=(x_1^2,\cdots,x_n^2)$ in this metric is a function $s[x^1,x^2]$, which is equivalent to the function 
\begin{equation*}
\overline{s}\left[x^1,x^2\right]=\frac{\sum_{i=1}^{n}|x_i^1-x_i^2|}{|x_n^1|^{\frac{\alpha}{2}}+|x_n^2|^{\frac{\alpha}{2}}+\sum_{i=1}^{n-1}|x_i^1-x_i^2|^{\frac{\alpha}{2}}}
\end{equation*}
in the sense that 
\begin{equation*}
s\leq C\overline{s} \qquad \mbox{and} \qquad \overline{s}\leq Cs
\end{equation*} 
for some constant $C>0$. For the parabolic problem we use the parabolic distance 
\begin{equation*}
s\left[(x^1,t_1),(x^2,t_2)\right]=s\left[x^1,x^2\right]+\sqrt{|t_1-t_2|}.
\end{equation*} 
In terms of this distance, we can define H\"older semi-norm and norm of continuous function $g$ on a compact subset $\mathbb{A}$ of the half-space $\{(x_1,\cdots,x_n,t):x_n\geq 0\}$:
\begin{equation*}
\begin{aligned}
\|g\|_{H_s^{\overline{\gamma}}(\mathbb{A})}&=\sup_{P_1\neq P_2\in\mathbb{A}}\frac{g(P_1)-g(P_2)}{s(P_1-P_2)^{\overline{\gamma}}},\\
\|g\|_{C_s^{\overline{\gamma}}(\mathbb{A})}=\|g\|_{C^0(\mathbb{A})}+&\|g\|_{H_s^{\overline{\gamma}}(\mathbb{A})}, \qquad \|g\|_{C^0(\mathbb{A})}=sup_{P\in\mathbb{A}}\left|g(P)\right|.
\end{aligned}
\end{equation*}
With these norms, the space $C_s^{2,\overline{\gamma}}(\mathbb{A})$ is the Banach space of all such functions with norm:
\begin{equation*}
\|g\|_{C^{2,\overline{\gamma}}_s(\mathbb{A})}=\|g\|_{C^{\overline{\gamma}}_s(\mathbb{A})}+\sum_{i=1}^{n}\|g_{x_i}\|_{C^{\overline{\gamma}}_s(\mathbb{A})}+\|g_t\|_{C^{\overline{\gamma}}_s(\mathbb{A})}+\sum_{1\leq i\leq j\leq n}^{n}\|x_n^{\alpha}g_{x_ix_j}\|_{C^{\overline{\gamma}}_s(\mathbb{A})}.
\end{equation*}
Imitating the case where the operators are defined on the half-space $\{(x_1,\cdots,x_n,t):x_n\geq 0\}$ we can define the distance function $s$ in $\Omega$. In the interior of $\Omega$ the distance will be equivalent to the standard Euclidean distance, while around any point $x_0\in\partial\Omega$, $s$ is defined as the pull back of the distance on the half space $\{(x_1,\cdots,x_n,t):x_n\geq 0\}$ through a map $\vp:\{(x_1,\cdots,x_n,t):x_n\geq 0\} \to \Omega$ that straightens the boundary of $\Omega$ near $x_0$.\\
\indent It can be easily shown that the distance between two points $P_1$ and $P_2$ in $\Omega$ is equivanlent to the function
\begin{equation*}
\overline{s}(P_1,P_2)=\frac{|P_1-P_2|}{d(P_1)^{\alpha}+d(P_2)^{\alpha}+\sum_{i=1}^{n-1}|P_1-P_2|^{\alpha}}
\end{equation*}
with $d=d(P)$ denoting the distance to the boundary of $\Omega$. The parabolic distance in the metric is equivalent to the function
\begin{equation*}
s\left[(P_1,t_1),(P_2,t_2)\right]=s\left[P_1,P_2\right]+\sqrt{|t_1-t_2|}.
\end{equation*} 
Suppose that $\mathbb{A}$ is a subset of $\Omega\times[0,\infty)$. As above, we denote by $C_{s}^{\overline{\gamma}}(\mathbb{A})$ the space of H\"older continuous functions on $\mathbb{A}$ with respect to the metric $s$ and by $C_s^{2,\overline{\gamma}}(\mathbb{A})$ the space of all functions $w$ on $\mathbb{A}$ such that $w_t$, $w_i$ and $d^{\alpha}w_{ij}$, with $i,j\in\{1,\cdots,n\}$ and with $d$ denoting the distance function to the boundary of $\Omega$, extend continuously up to the boundary of $\mathbb{A}$ and the extensions are H\"older continuous on $\mathbb{A}$ of class $C_{s}^{\overline{\gamma}}(\mathbb{A})$. Then, they are both Banach spaces under the norm $\|w\|_{C_s^{\overline{\gamma}}(\mathbb{A})}$ and 
\begin{equation*}
\|w\|_{C^{2,\overline{\gamma}}_s(\mathbb{A})}=\|w\|_{C^{\overline{\gamma}}_s(\mathbb{A})}+\sum_{i=1}^{n}\|w_{x_i}\|_{C^{\overline{\gamma}}_s(\mathbb{A})}+\|w_t\|_{C^{\overline{\gamma}}_s(\mathbb{A})}+\sum_{1\leq i\leq j\leq n}^{n}\|d^{\alpha}w_{x_ix_j}\|_{C^{\overline{\gamma}}_s(\mathbb{A})}.
\end{equation*}
Under this metric, it is known that the problem \eqref{eq-cases-main-of-this-paper-section-PMT} has the $C^{2,\overline{\gamma}}_s$ solution for a short time interval $(0,T)$. we now state the short time existence for \eqref{eq-cases-main-of-this-paper-section-PMT}.
\begin{thm}[Theorem 1.1 in \cite{KL}]\label{thm-short}
Let $d$ be the distance to the boundary $\Omega$. If the functions $f^0$, $Df^0$ and $d^{\alpha}D^2f^0$, restricted to the compact domain $\Omega$, extended continuously up to the boundary of $\Omega$, with extensions which are H\"older continuous on $\Omega$ of class $C^{\overline{\gamma}}(\Omega)$, for some $\overline{\gamma}>0$ and $Df_0\neq 0$ along $\partial\Omega$, then there exists a number $T>0$ for which the initial value problem
\begin{equation*}
\begin{cases}
f_t=mf^{\alpha}\La f, \qquad (x,t)\in\Omega\times(0,\infty),\\
f(x,0)=f^0(x), \qquad x\in\Omega\\
f(x,t)=0, \qquad x\in\partial\Omega\times[0,\infty)
\end{cases}
\end{equation*}
admits a solution $f$ which is $C_s^{2,\overline{\gamma}}$-smooth up to the boundary, when $0<t<T$.
\end{thm}

\section{Global $C^{1}_{s}$-estimate}
\setcounter{equation}{0}
\setcounter{thm}{0}

This section is devoted to prove some properties of the solution $f$ to \eqref{eq-cases-main-of-this-paper-section-PMT}, including global estimates, gradient estimate on the boundary, etc... . To get them we construct sub and super solutions to the problem \eqref{eq-cases-main-of-this-paper-section-PMT}. Now, we first deal with the $L^{\infty}$ estimate of $f$.
\begin{lem}[$L^{\infty}$-estimate]\label{L-infty-estimate}
Let $d$ be the distance to the boundary of $\partial\Omega$ and let $f$ be the solution of \eqref{eq-cases-main-of-this-paper-section-PMT} with initial data $f^0$ such that
\begin{equation*}
f^0,\,\, Df^0,\,\, d^{\alpha}D^2f^0 \,:\qquad \mbox{continuous up to the boundary $\partial\Omega$}.
\end{equation*}
There are constants $c>0$ and $C\leq \infty$ such that
         \begin{enumerate}
         \item  $$\|f\|_{L^{\infty}(\Omega)}\leq \min\left\{\frac{C}{(1+t)^{1/\alpha}}, \|f_0\|_{L^{\infty}(\Omega)}\right\}, \qquad \left(\alpha=1-\frac{1}{m}\right)$$      
         \item There is a ball $B_{\delta_0}\in\Omega$ such that 
                     $$\inf_{B_{\delta_0}}f(x,t)\geq \frac{c}{(1+t)^{1/\alpha}}.$$                             
         \end{enumerate}

\end{lem}

\begin{proof}
i) We take a ball $B_R=B_R(0)$ of radius $R$ strictly containing $\Omega$, especially we assume that $\Omega\subset B_{\frac{R}{2}}$, and consider the function $z(x,t)$ defined in $B_R\times(0,\infty)$ by
\begin{equation*}
z(x,t)=\frac{A(R^2-|x|^2)}{(1+t)^{\frac{1}{\alpha}}}
\end{equation*}
for suitable constant $A$ to be chosen presently. Since $z$ is positive in $B_R\times(0,\infty)$, we have
\begin{equation*}
f(x,t)=0<z(x,t) \qquad \mbox{on $\partial\Omega\times(0,\infty)$}.
\end{equation*}
In addition, if we choose $A$ larger than $\frac{4\|f_0\|_{L^{\infty}(\Omega)}}{R^2}$, we get
\begin{equation*}
f_0(x)\leq z(x,0) \qquad \mbox{in $\Omega$}.
\end{equation*}
Finally, we will obtain the inequality $mz^{\alpha}\La z-z_t\leq 0$ in $\Omega$ whenever
\begin{equation*}
A>\left[\frac{R^{2(1-\alpha)}}{2mn\alpha}\right]^{\frac{1}{\alpha}}.
\end{equation*}
With this choice, the comparison principle implies that 
\begin{equation*}
f(x,t)\leq z(x,t)=\frac{A(R^2-|x|^2)}{(1+t)^{\frac{1}{\alpha}}}\leq \frac{AR^2}{(1+t)^{\frac{1}{\alpha}}}=\frac{C}{(1+t)^{\frac{1}{\alpha}}}\qquad \mbox{in $\Omega\times(0,\infty)$}.
\end{equation*}
On the other hand, one can easily check $f(x,t)\leq \|f_0\|_{L^{\infty}(\Omega)}$ by comparison principle. Hence the proof of (i) is finished.\\
ii) We first suppose that the initial data $u_0=f_0^{\frac{1}{m}}$ satisfies the following condition,
\begin{equation*}
(m-1)\La u_0^m\leq-u_0 \qquad \forall x\in\Omega.
\end{equation*} 
Then, the initial data $u_0$ is controlled from below by the solution $g(x)$ of
\begin{equation*}
\begin{cases}
\begin{aligned}
(m-1)\La g^m+&g=0 \qquad \mbox{in $\Omega$}\\
g=&0 \qquad \quad \mbox{on $\partial\Omega$},
\end{aligned}
\end{cases}
\end{equation*}
i.e.,
\begin{equation*}
u_0(x)\geq g(x).
\end{equation*}
In addition, the function
\begin{equation*}
\frac{g(x)}{(1+t)^{\frac{1}{m-1}}}
\end{equation*}
is also a solution of porous medium equation \eqref{eq-main} with the initial data $g$. Hence, by the comparison principle, we have
\begin{equation*}
u(x,t) \geq \frac{g(x)}{(1+t)^{\frac{1}{m-1}}}.
\end{equation*}
Therefore, for any ball $B_{\delta_0}\in\Omega$, we have
\begin{equation*}
\inf_{x\in B_{\delta_0}}f(x,t) \geq \frac{\inf_{x\in B_{\delta_0}}g(x)}{(1+t)^{\frac{m}{m-1}}}.
\end{equation*}
Next, we denote by $\Omega_{u_0}$ the set
\begin{equation*}
\Omega_{u_0}=\{x\in\Omega : (m-1)\La u_0^m>-u_0\}
\end{equation*}
and we assume that $\Omega_{u_0}\neq\emptyset$. Let's define the function $\overline{u}_0$ such that
\begin{equation*}
\overline{u}_0=u_0 \quad \mbox{on $\Omega_{u_0}$}
\end{equation*}
and 
\begin{equation*}
(m-1)\La\overline{u}_0^m>-\overline{u}_0 \qquad \mbox{on $\Omega$}.
\end{equation*}
Then, by comparison principle for elliptic equation, $\overline{u}_0\leq u_0$ in $\Omega$.
We also let $\overline{u}$ be the solution of \eqref{eq-main} with the initial data $u_0$ being replaced by $\overline{u}_0$. Then, by comparison principle for parabolic equation, we have
\begin{equation*}
u\geq \overline{u} \qquad \mbox{in $\Omega$}.
\end{equation*}
On the other hand, by the Problem 8.1(i) in \cite{Va1}, $\overline{u}$ satisfies
\begin{equation*}
\overline{u}_t\geq -\frac{\overline{u}}{(m-1)(1+t)}.
\end{equation*}
Thus, by the Gronwell's inequality, we have
\begin{equation*}
\overline{u}(x,t)\geq \frac{\overline{u}_0(x)}{(1+t)^{\frac{1}{m-1}}}.
\end{equation*}
Since $\Omega_{u_0}$ is open set, there exists a ball $B_{\delta_0}\in\Omega_{u_0}$. Hence
\begin{equation*}
u(x,t)\geq \overline{u}(x,t)\geq \frac{\overline{u}_0(x)}{(1+t)^{\frac{1}{m-1}}}=\frac{u_0(x)}{(1+t)^{\frac{1}{m-1}}} \qquad \mbox{on $B_{\delta_0}$}.
\end{equation*}
Therefore
\begin{equation*}
\inf_{x\in B_{\delta_0}}f(x,t)=\inf_{x\in B_{\delta_0}}u^m(x,t)\geq \frac{\inf_{x\in B_{\delta_0}}u_0^m(x)}{(1+t)^{\frac{m}{m-1}}}=\frac{\inf_{x\in B_{\delta_0}}f_0(x)}{(1+t)^{\frac{m}{m-1}}}
\end{equation*}
, which implies the conclusion.
\end{proof}
Next property is the gradient estimate which will play an important role to show $C_{s}^{1,\overline{\gamma}}$ continuity of solution.

\begin{lem}[Gradient estimate]\label{lem-Gradient-estimate}
Let $d$ be the distance to the boundary of $\partial\Omega$ and let $f$ be the solution of \eqref{eq-cases-main-of-this-paper-section-PMT} with initial data $f^0$ such that
\begin{equation*}
f^0,\,\, Df^0,\,\, d^{\alpha}D^2f^0 \,:\qquad \mbox{continuous up to the boundary $\partial\Omega$}.
\end{equation*}
There are uniform constant $0<C_o<\infty$ such that
\begin{equation}\label{eq-gradient-bound-of-f-1}
\|\D f\|_{L^{\infty}( \Omega)}<C_0\frac{\|f_0\|_{C^1(\Omega)}}{(1+t)^{1/\alpha}}
\end{equation}
\end{lem}
\begin{proof}
By the Theorem \ref{thm-short}, there exist constants $t_0>0$ and $C_1>0$ such that
\begin{equation}\label{eq-grad-bound-for-short-time}
|\nabla f(x,t)|<C_1 \qquad \forall x\in\Omega,\,\,0\leq t<t_0.
\end{equation}
Thus, \eqref{eq-gradient-bound-of-f-1} holds for a short time.\\
\indent Let $\psi$ be the solution of 
\begin{equation*}
\begin{cases}
\La\psi+\frac{1}{m-1}\psi^{\frac{1}{m}}=0 \qquad \mbox{in $\Omega$}\\
\quad \psi(x)=0 \qquad \qquad \mbox{on $\partial\Omega$}.
\end{cases}
\end{equation*}
Since $0<|\nabla f_0|<\infty$ on $\partial\Omega$, we can select constants $0<c_1<1<c_2<\infty$ such that 
\begin{equation*}
c_1\psi(x)\leq f_0(x)\leq c_2\psi(x), \qquad \mbox{in $\Omega$}.
\end{equation*}
Then, by the comparison principle, we get
\begin{equation*}
\frac{c_1\psi(x)}{(1+t)^{\frac{1}{\alpha}}}\leq f(x,t) \leq \frac{c_2\psi(x)}{(1+t)^{\frac{1}{\alpha}}}, \qquad x\in\Omega.
\end{equation*}
Since $\|\psi(x)\|_{C^{1}(\partial\Omega)}<\infty$, we also have
\begin{equation}\label{eq-bound-of-gradient-of-f-on-boundary}
|\nabla f(x,t)|\leq \frac{c_2\|\psi(x)\|_{C^{1}(\partial\Omega)}}{(1+t)^{\frac{1}{\alpha}}}\leq \frac{c_2\|f_0(x)\|_{C^{1}(\Omega)}}{c_1(1+t)^{\frac{1}{\alpha}}}<\infty, \qquad \forall 0\leq t<\infty, x\in\partial\Omega.
\end{equation}
Denoting by $\Omega_{\sigma}$, for $\sigma>0$, the set
\begin{equation*}
\Omega_{\sigma}=\{x\in\Omega:\textbf{dist}(x,\partial\Omega)\geq \sigma\}.
\end{equation*}
Then, by \eqref{eq-bound-of-gradient-of-f-on-boundary}, there exist constants $\sigma_0>0$ and $c_3>0$ such that 
\begin{equation}\label{eq-bound-of-gradient-of-f-on-near-boundary}
|\nabla f(x,t)|\leq \frac{c_3\| f_0(x)\|_{C^{1}(\Omega)}}{(1+t)^{\frac{1}{\alpha}}}<\infty \qquad \forall 0\leq t<\infty, x\in \Omega\bs\Omega_{\sigma_0}.
\end{equation}
To finish the proof, let us define the family of rescaled functions
\begin{equation*}
f_k(x,t)=k^{\frac{1}{\alpha}}f(x,kt)
\end{equation*}
in $\Omega_{\sigma_0}$ with parameter $k>0$. Then, they are again solutions of
\begin{equation}\label{eq-of-f-for-gradient-1}
f_t=mf^{\alpha}\La f.
\end{equation}
Since $f>0$ in $\Omega_{\sigma_0}$, the coefficient $mf^{\alpha}$ is bounded above and below. Hence the equation \eqref{eq-of-f-for-gradient-1} becomes uniformly parabolic and the solutions are smooth with derivatives locally bounded in terms of the bounds for $f$. We conclude that there exist a constant $c_4>0$ such that
\begin{equation*}
|\nabla f_k(x,1)|<c_4 \qquad x\in \Omega_{\sigma_0}.
\end{equation*}
This means that
\begin{equation}\label{eq-inequality-of-gradi-f-1}
|\nabla f(x,k)|<\frac{c_4}{k^{\frac{1}{\alpha}}} \qquad x\in \Omega_{\sigma_0}.
\end{equation}
For some constant $c_5>0$, putting $t=k$ and $c_4=c_5\|f_0\|_{C^1(\Omega)}$ in inequality \eqref{eq-inequality-of-gradi-f-1}. Then
\begin{equation}\label{eq-inequality-of-gradi-f-2}
|\nabla f(x,t)|<\frac{c_5\|f_0\|_{C^1(\Omega)}}{t^{\frac{1}{\alpha}}} \qquad x\in \Omega_{\sigma_0}.
\end{equation}
Hence, we get
\begin{equation}\label{eq-grad-bound-for-long-time}
|\nabla f(x,t)|<\frac{c_6\|f_0\|_{C^1(\Omega)}}{(1+t)^{\frac{1}{\alpha}}}, \qquad \forall t\geq t_0\,\,x\in \Omega_{\sigma_0}
\end{equation}
where $c_6=c_5\left(1+\frac{1}{t_0}\right)^{\frac{1}{\alpha}}$. By \eqref{eq-grad-bound-for-short-time}, \eqref{eq-bound-of-gradient-of-f-on-near-boundary} and \eqref{eq-grad-bound-for-long-time}, \eqref{eq-gradient-bound-of-f-1} holds for all $t>0$.
\end{proof}
    
Finally in this section, we will show the non degeneracy of solution $f$ to \eqref{eq-cases-main-of-this-paper-section-PMT} near the boundary. This estimate quarantees the uniformly ellipticity of coefficients $a^{kl}$ in \eqref{eq-for-sol-h-x-i-of-fixed-bondary-problem-in-intro}.

\begin{lem}[Non-degeneracy estimate]\label{lem-Non-degeneracy-estimate}
Let $d$ be the distance to the boundary of $\partial\Omega$ and let $f$ be the solution of \eqref{eq-cases-main-of-this-paper-section-PMT} with initial data $f^0$ such that
\begin{equation*}
f^0,\,\, Df^0,\,\, d^{\alpha}D^2f^0 \,:\qquad \mbox{continuous up to the boundary $\partial\Omega$}.
\end{equation*}
There are uniform constant $0<c_1<\infty $ such that
          $$ \frac{c_0}{(1+t)^{1/\alpha}} < \|\D f\|_{L^{\infty}( \partial\Omega)} $$                   
\end{lem} 

\begin{proof}
By the Lemma \ref{L-infty-estimate} (ii), there exists a ball $B_{\delta_0}\subset\Omega$ such that
\begin{equation*}
\inf_{x\in B_{\delta_0}}f(x,t)\geq \frac{c_0}{(1+t)^{\frac{1}{\alpha}}}.
\end{equation*}
For fixed $t>0$, let $v$ be the solution of the problem
\begin{equation*}
\begin{cases}
\begin{aligned}
\La v(x,t)&=0 \qquad \qquad \mbox{in $\Omega\bs B_{\delta_0}$}\\
v(x,t)&=c_0\qquad \qquad  \mbox{on $\partial B_{\delta_0}$}\\
v(x,t)&=0 \qquad \qquad\mbox{on $\partial\Omega$}.
\end{aligned}
\end{cases}
\end{equation*}
Then, the function $V(x,t)=\frac{v(x,t)}{(1+t)^{\frac{1}{\alpha}}}$ satisfies
\begin{equation*}
V^{\alpha}\La V-V_t=\frac{v}{\alpha(1+t)^{1+\frac{1}{\alpha}}} \geq 0.
\end{equation*}
In addition, we have
\begin{equation*}
V(x,t)=\frac{v(x,t)}{(1+t)^{\frac{1}{\alpha}}}=\frac{c_0}{(1+t)^{\frac{1}{\alpha}}} \leq f(x,t) \quad \mbox{on $\partial B_{\delta_0}$}
\end{equation*}
 and
\begin{equation*} 
V(x,t)=\frac{v(x,t)}{(1+t)^{\frac{1}{\alpha}}}=0=f(x,t) \quad \mbox{on $\partial\Omega$}.
\end{equation*}
By the comparison principle, we have
\begin{equation*}
f(x,t)\geq V(x,t) \qquad \mbox{in $\Omega\bs B_{\delta_0}$}.
\end{equation*}
By the Hopf's inequality for the harmonic equation, there exists some constant $c_1>0$ such that
\begin{equation*}
\frac{\partial v}{\partial \nu} \leq -c_1<0 \qquad \mbox{on $\partial\Omega$}
\end{equation*}
for the outer normal direction $\nu$ to $\partial\Omega$. Therefore, 
\begin{equation*}
\|\nabla f\|_{L^{\infty}(\partial\Omega)}\geq \|\nabla V\|_{L^{\infty}(\partial\Omega)}=\frac{\|\nabla v\|_{L^{\infty}(\partial\Omega)}}{(1+t)^{\frac{1}{\alpha}}} \geq \frac{c_1}{(1+t)^{\frac{1}{\alpha}}}
\end{equation*}
and lemma follows.
\end{proof}

\section{H\"older estimate I}
\setcounter{equation}{0}
\setcounter{thm}{0}

In the previous section, we obtained the global $C_s^1$ regularity of solution $f$ to the problem \eqref{eq-cases-main-of-this-paper-section-PMT}. By the relation between $f$ and $h$, we can say the same story on the solution $h$ of \eqref{eq-for-sol-h-of-fixed-bondary-problem-in-intro}. Hence, we can have basic informations for coefficients $a^{kl}$ and solution $w$ in \eqref{eq-for-sol-h-x-i-of-fixed-bondary-problem-in-intro}, including boundedness of solution and coefficients. With this basic properties, we devote this section and next one for solving the missing step, \eqref{eq-missing step for the regularity}, for the regularity theory.\\ 
\indent Let $H$ be the half space $\{x=(x_1,\cdots,x_n)\in\R^n:x_n>0\}$. We are going to show H\"older estimate on a solution $w$ of the equation
\begin{equation}\label{eq-general-of-tilde-h-x-n-1}
w_{t}=x_n^{\alpha}\nabla_i\left(a^{ij}\nabla_j w\right)+x_n^{\alpha}g \qquad \mbox{in $H$}
\end{equation}
for a forcing term $g$. Assume that the coefficients $a^{ij}(x,t)$ are measurable functions and satisfy
\begin{equation}\label{eq-assumption-for-coefficients-1}
\lambda|\xi|^2\leq a^{ij}(x,t)\xi_i\xi_j\leq \Lambda|\xi|^2, \qquad (i,j=1,\cdots,n).
\end{equation}
In addition, we suppose that the forcing term $g$ satisfies
\begin{equation}\label{eq-assumption-for-coefficients-2}
|g(x)|\leq C|w(x)|
\end{equation}
for some constant $C<\infty$.\\
\indent For the H\"older estimates of the solution $w$ to \eqref{eq-general-of-tilde-h-x-n-1}, we need the following two inequalities. The first one is a weighted version of Sobolev's inequality. Let $C^{\infty}_0(H)$ be the space of restriction of functions in $C^{\infty}_0(\R^n)$ to $H$. 
\begin{lem}[See Theorem 4.2.2 in \cite{Ko}]\label{lem-weighted-sobolev-inequality-in-H-Koch-paper}
Let $1\leq p\leq q<\infty$, $s>-\frac{1}{p}$ and we assume that $\sigma\leq 1$ satisfies
\begin{equation*}
\frac{1}{q}-\frac{1-\sigma}{n}=\frac{1}{p}.
\end{equation*}
Then
\begin{equation*}
\left(\int_{H}x_n^{sp}u^p\,dx\right)^{\frac{1}{p}}\leq c\left(\int_{H}x_n^{(s+\sigma)q}|\nabla u|^{q}\,dx\right)^{\frac{1}{q}}.
\end{equation*}
for the closure of $C^{\infty}_0(H)$.
\end{lem}

The next inequality is the parabolic version of Sobolev's inequality with weight: 

\begin{lem}\label{lem-reviews-CS}
Let $Q=B\times(a,b)$ is cylinder in $H\times(0,1)$ and $f\in C_0^{\infty}(Q)$. Then there are constant $C>0$ such that if $l_1=\frac{n+2-2\alpha}{n-\alpha}$ we have
\begin{equation*}
\left(\int_{a}^{b}\int_{B}x_n^{-\alpha}|f|^{2l_1}\,dxdt\right)^{\frac{1}{l_1}}\leq C\left[\left(\sup_{(a,b)}\int_{B}x_n^{-\alpha}|f|^2\,dx\right)+\int_{a}^{b}\int_{B}|\nabla f|^2\,dxdt\right].
\end{equation*}
\end{lem}

\begin{proof}
By the H\"older inequality,
\begin{equation*}
\left(\int_{B}x_n^{-\alpha}|f|^{2l_1}\,dx\right)^{\frac{1}{l_1}}\leq C_1\left(\int_{B}x_n^{-\alpha}|f|^2\,dx\right)^{\frac{l_1-1}{l_1}}\left(\int_Bx_n^{-\alpha}|f|^{\frac{2}{2-l_1}}\,dx\right)^{\frac{2-l_1}{l_1}}.
\end{equation*}
for some costant $C_1>0$. By the Lemma \ref{lem-weighted-sobolev-inequality-in-H-Koch-paper}, it follows that
\begin{equation*}
\left(\int_{B}x_n^{-\alpha}|f|^{2l_1}\,dx\right)^{\frac{1}{l_1}}\leq C_2\left(\int_{B}x_n^{-\alpha}|f|^2\,dx\right)^{1-\frac{1}{l_1}}\left(\int_{B}|\nabla f|^2\,dx\right)^{\frac{1}{l_1}}.
\end{equation*}
for some costant $C_2>0$. Now, taking the the $l_1$ power and integrating in $(a,b)$, we get
\begin{equation*}
\left(\int_{a}^{b}\int_{B}x_n^{-\alpha}|f|^{2l_1}\,dxdt\right)^{\frac{1}{l_1}}\leq C_3\left[\left(\sup_{(a,b)}\int_{B}x_n^{-\alpha}|f|^2\,dx\right)+\int_a^b\int_{B}|\nabla f|^2\,dxdt\right]
\end{equation*}
for some constant $C_3>0$ and the lemma follows.
\end{proof}

Define the balls $B_{r}$ and $B_{r}^+$ of radius $r$ around $x=x_0$ to be the sets
\begin{equation*}
B_r(x_0)=\{x\in\R^n:|x-x_0|<r\} \qquad \mbox{and } \qquad B_{r}^{+}(x_0)=B_{r}(x_0)\cap\{ x_n> 0\}.
\end{equation*}
We let $B_r$ be the ball around the point $x=0$. We define the round cubes $Q_r$ of radius $r$ around $(x,t)=(0,1)$ to be the sets
\begin{equation*}
Q_r=B_r\times(1-r^{2-\alpha},1).
\end{equation*}
We also define the general round cubes: 
\begin{equation*}
Q_r(x,t)=Q_r+(x,t)\qquad \mbox{and}\qquad Q_r^+(x,t)=Q_{r}(x,t)\cap\{ x_n> 0\}.
\end{equation*}
\indent Let us give the Harnack inequality first.

\begin{lem}[Harnack's Inequality]\label{lem-Harnack-s-Inequality}
Let $w$ be a solution of equation \eqref{eq-general-of-tilde-h-x-n-1} defined in $B^+_1\times(0,1)$ with conditions \eqref{eq-assumption-for-coefficients-1} and \eqref{eq-assumption-for-coefficients-2}. Let $|w|\leq M$ on $Q^+_{1}$ and $\hat{h}=w+M+1\geq 1$. Then,
\begin{equation*}
\max_{Q^+_{\frac{1}{8}}\left(0,-\frac{3}{8}\right)}\hat{h} \leq C\min_{Q^+_{\frac{1}{8}}}\hat{h}
\end{equation*}
\end{lem}

\begin{proof}
Let $\mathcal{Q}^+=\mathcal{B}^+\times(s_2,s_1)\in H\times[0,\infty)$ be an round cube in $Q_1^+$ and we take $\phi(x,t)=\eta^{2}(x,t)\hat{h}^{\delta}(x,t)$ as test function, where $\eta\in C^{\infty}(\mathcal{Q}^+)$ with $\eta=0$ on $\partial_p\mathcal{Q}^+\cap\{x_n>0\}$. Since $w$ is solution of \eqref{eq-general-of-tilde-h-x-n-1} in $Q^+_1$, $\hat{h}$ satisfies
\begin{equation}\label{eq-equaiton-for-hat-h-20}
\frac{\hat{h}_{t}}{x_n^{\alpha}}=\nabla_i\left(a^{ij}\nabla_j \hat{h}\right)+g \qquad \mbox{in $H$}.
\end{equation}
Hence, multiplying \eqref{eq-equaiton-for-hat-h-20} by the test function $\phi(x,t)$, we arrive at
\begin{equation*}
\begin{aligned}
&\frac{4\lambda|\delta|}{|1+\delta|^2}\iint_{\mathcal{Q}^+}\left|\nabla\left(\eta\,\hat{h}^{\frac{1+\delta}{2}}\right)\right|^2\,dxdt+\frac{1}{|1+\delta|}\left(\sup_{s_1\leq t\leq s_2}\int_{\mathcal{B}^+}x_n^{-\alpha}\eta^2\hat{h}^{1+\delta}(\cdot,t)\,dx\right) \\
&\qquad \qquad \leq \frac{8\Lambda(1+3|\delta|)}{|1+\delta|^2}\iint_{\mathcal{Q}^+}\hat{h}^{\frac{1+\delta}{2}}\,|\nabla\eta|\,\left|\nabla\left(\eta\,\hat{h}^{\frac{1+\delta}{2}}\right)\right|\,dxdt+\frac{8\Lambda(1+2|\beta|)}{|1+\delta|^2}\iint_{\mathcal{Q}^+}\hat{h}^{1+\delta}\,|\nabla\eta|^2\,dxdt\\
&\qquad \qquad \quad +2\iint_{\mathcal{Q}^+}\eta^2\hat{h}^{\delta}\,\left|\overline{g}\right|\,dxdt+\frac{4}{|1+\delta|}\iint_{\mathcal{Q}^+}x_n^{-\alpha}\hat{h}^{1+\delta}\eta|\eta_t|\,dxdt
\end{aligned}
\end{equation*}
for $\delta>0$ or $\delta<-1$. Use of Young's inequality yields
\begin{equation}\label{eq-aligned-after-applying-first-young-inequality}
\begin{aligned}
&\sup_{s_1\leq t\leq s_2}\int_{\mathcal{B}}x_n^{-\alpha}\eta^2\hat{h}^{1+\delta}(\cdot,t)\,dx+\iint_{\mathcal{Q}^+}\left|\nabla\left(\eta\,\hat{h}^{\frac{1+\delta}{2}}\right)\right|^2\,dxdt\\
&\qquad \qquad \leq C_1\left[\iint_{\mathcal{Q}^+}\left(|\nabla\eta|^2+x_n^{-\alpha}\eta|\eta_t|\right)\,\left(\hat{h}^{\frac{1+\delta}{2}}\right)^2\,dxdt+\iint_{\mathcal{Q}^+}\eta^2\hat{h}^{\delta}\left|g\right|\,dxdt\right]
\end{aligned}
\end{equation}
for some constant $C_1>0$.\\
\indent We first consider the case $\delta>0$. Let $r_1$ and $r_2$ be such that $\frac{1}{8}\leq r_2\leq r_1\leq\left(\frac{1}{8}\right)^{\frac{1}{2-\alpha}}$, choosing $\eta$ in such a way that $\eta(x,t)=1$ in $Q^+_{r_2}\left(0,-\frac{3}{8}\right)$, $\eta(x,t)=0$ on $\partial_pQ^+_{r_1}\left(0,-\frac{3}{8}\right)\cap H$, $0\leq \eta \leq 1$ in $Q_{r_1}^+\left(0,-\frac{3}{8}\right)$, $|\nabla\eta|\leq \frac{C_2}{r_1-r_2}$ and $|\eta_t|\leq \frac{C_2}{(r_1-r_2)^{2-\alpha}}$ for some constant $C_2>0$. Then, for $\mathcal{Q}^+=Q_{\left(\frac{1}{8}\right)^{\frac{1}{2-\alpha}}}^+\left(0,-\frac{3}{8}\right)=B^+_{\left(\frac{1}{8}\right)^{\frac{1}{2-\alpha}}}\times\left(\frac{1}{4},\frac{5}{8}\right)$,
\begin{equation}\label{eq-first-term-of-right-hand-side-of-upper-eq}
\begin{aligned}
\iint_{\mathcal{Q}^+}\left(|\nabla\eta|^2+x_n^{-\alpha}|\eta||\eta_t|\right)\,\left(\hat{h}^{\frac{1+\delta}{2}}\right)^2\,dxdt &\leq \frac{C_3}{|r_1-r_2|^{2-\alpha}}\iint_{Q^+_{r_1}\left(0,-\frac{3}{8}\right)}x_n^{-\alpha}\left(\hat{h}^{\frac{1+\delta}{2}}\right)^2\,dxdt.
\end{aligned}
\end{equation}
for some constant $C_3>0$. On the other hand, by \eqref{eq-assumption-for-coefficients-2}, the second term of the right hand side in \eqref{eq-aligned-after-applying-first-young-inequality} is changed to
\begin{equation}\label{eq-second-term-of-right-hand-side-of-upper-eq}
\begin{aligned}
\int_{\mathcal{Q}^+}\left(\eta^2\hat{h}^{\delta}\right)\left|g\right|\,dxdt&\leq C_4\int_{\mathcal{Q}^+}\left(\eta^2\hat{h}^{1+\delta}\right)\left|w\right|\,dxdt\leq C_4\int_{Q^+_{r_1}\left(0,-\frac{3}{8}\right)}\hat{h}^{1+\delta}\,dxdt
\end{aligned}
\end{equation}
for some constants $C_4>0$. Applying \eqref{eq-first-term-of-right-hand-side-of-upper-eq}, \eqref{eq-second-term-of-right-hand-side-of-upper-eq} to \eqref{eq-aligned-after-applying-first-young-inequality}, we can get
\begin{equation}\label{eq-aligned-after-applying-second-young-inequality}
\begin{aligned}
\sup_{\frac{1}{4}\leq t\leq \frac{5}{8}}\int_{\mathcal{B}^+}x_n^{-\alpha}\eta^2\hat{h}^{1+\delta}(\cdot,t)\,dx&+\iint_{\mathcal{Q}^+}\left|\nabla\left(\eta\hat{h}^{\frac{1+\delta}{2}}\right)\right|^2\,dxdt \\
&\qquad \qquad \qquad \leq \frac{C_5}{|r_1-r_2|^{2-\alpha}}\iint_{Q^+_{r_1}\left(0,-\frac{3}{8}\right)}x_n^{-\alpha}\left(\hat{h}^{\frac{1+\delta}{2}}\right)^2\,dxdt
\end{aligned}
\end{equation}
for some constant $C_5>0$. By the Lemma \ref{lem-reviews-CS}, there exists $l_1>1$ such that
\begin{equation}\label{eq-aligned-lower-bound-of-sup-and-gradient-of-hat-h-by-L-to-l-1-norm}
\begin{aligned}
&\left(\iint_{\mathcal{Q}^+}x_n^{-\alpha}\left(\hat{h}^{\frac{1+\delta}{2}}\right)^{2l_1}\,dxdt\right)^{\frac{1}{l_1}}\\
&\qquad \qquad \leq C_6\left[\sup_{\frac{1}{4}\leq t\leq \frac{5}{8}}\int_{\mathcal{B}^+}x_n^{-\alpha}\hat{h}^{1+\delta}(\cdot,t)\,dx+\iint_{\mathcal{Q}^+}\left|\nabla\left(\hat{h}^{\frac{1+\delta}{2}}\right)\right|^2\,dxdt\right]
\end{aligned}
\end{equation}
for some constant $C_6>0$. Then, by \eqref{eq-aligned-after-applying-second-young-inequality} and \eqref{eq-aligned-lower-bound-of-sup-and-gradient-of-hat-h-by-L-to-l-1-norm},
\begin{equation}\label{eq-aligned-lower-bound-of-sup-and-gradient-of-hat-h-by-L-to-l-0-norm}
\begin{aligned}
\left(\iint_{Q_{r_2}^+\left(0,-\frac{3}{8}\right)}x_n^{-\alpha}\left(\hat{h}^{\frac{1+\delta}{2}}\right)^{2l_1}\,dxdt\right)^{\frac{1}{l_1}} \leq  \frac{C_{7}}{|r_1-r_2|^{2-\alpha}}\int_{Q^+_{r_1}\left(0,-\frac{3}{8}\right)}x_n^{-\alpha}\left(\hat{h}^{\frac{1+\delta}{2}}\right)^2\,dxdt
\end{aligned}
\end{equation}
for some constant $C_{7}>0$. Taking the $\frac{1}{1+\delta}$th root on each side of \eqref{eq-aligned-lower-bound-of-sup-and-gradient-of-hat-h-by-L-to-l-0-norm}, we have
\begin{equation*}
\|\hat{h}\|_{L^{\left(1+\delta\right)l_1}\left(Q^+_{r_2}\left(0,-\frac{3}{8}\right),x_n^{-\alpha}\right)}\leq \frac{C_{8}}{|r_1-r_2|^{\frac{2-\alpha}{1+\delta}}}\|\hat{h}\|_{L^{1+\delta}\left(Q^+_{r_1}\left(0,-\frac{3}{8}\right),x_n^{-\alpha}\right)}
\end{equation*}
for some constant $C_{8}>0$. Define $\gamma_1=p_0>0$, $\gamma_i=l_1\gamma_{i-1}$ and $r_i=\frac{1}{8}+\left(\left(\frac{1}{8}\right)^{\frac{1}{2-\alpha}}-\frac{1}{8}\right)^{i}$. Then, the previous inequality becomes
\begin{equation*}
\|\hat{h}\|_{L^{\gamma_{i+1}}\left(Q^+_{r_{i+1}}\left(0,-\frac{3}{8}\right),x_n^{-\alpha}\right)}\leq C_{8}\left(\frac{8}{9}\right)^{\frac{2-\alpha}{\gamma_i}}\left(8^{i}\right)^{\frac{1}{\gamma_i}}\|\hat{h}\|_{L^{\gamma_i}\left(Q^+_{r_i}\left(0,-\frac{3}{8}\right),x_n^{-\alpha}\right)}
\end{equation*}
Therefore, iteration yields
\begin{equation}\label{eq-upperbound-of-max-overline-h}
\max_{Q^+_{\frac{1}{8}}\left(0,-\frac{3}{8}\right)}\hat{h}\leq C_{9}\|\hat{h}\|_{L^{p_0}\left(B^+_{\left(\frac{1}{8}\right)^{\frac{1}{2-\alpha}}}\times\left(\frac{1}{2},\frac{5}{8}\right),x_n^{-\alpha}\right)}, \quad (C_{9}>0).
\end{equation}
For the other case $\delta<-1$, we also take $\phi(x,t)=\eta^{2}(x,t)\hat{h}^{\delta}(x,t)$ as test function, choosing $\eta$ in such a way that $\eta(x,t)=1$ in $Q^+_{r_2}$, $\eta(x,t)=0$ on $\partial_pQ^+_{r_1}\cap H$, $0\leq \eta \leq 1$ in $Q_{r_1}^+$, $|\nabla\eta|\leq \frac{C_{10}}{r_1-r_2}$ and $|\eta_t|\leq \frac{C_{10}}{(r_1-r_2)^{2-\alpha}}$ for some constant $C_{10}>0$ and $\frac{1}{8}\leq r_2\leq r_1\leq\left(\frac{1}{8}\right)^{\frac{1}{2-\sigma}}$. Then, by the similar computation above, we have
\begin{equation*}
\|\hat{h}\|_{L^{(1+\delta)l_0}(Q^+_{r_2},x_n^{-\alpha})}\geq \frac{C_{11}}{|r_1-r_2|^{\frac{2-\alpha}{1+\delta}}}\|\hat{h}\|_{L^{1+\delta}(Q^+_{r_1},x_n^{-\alpha})}
\end{equation*}
for some constant $C_{11}>0$. Define $\gamma_0=-p_0<0$, $\gamma_i=l_1\gamma_{i-1}$ and $r_i=\frac{1}{8}+\left(\left(\frac{1}{8}\right)^{\frac{1}{2-\sigma}}-\frac{1}{8}\right)^{i}$. Then, the previous inequality becomes
\begin{equation*}
\|\hat{h}\|_{L^{\gamma_{i+1}}(Q^+_{r_{i+1}},x_n^{-\alpha})}\geq C_{11}\left(\frac{8}{9}\right)^{\frac{2-\alpha}{\gamma_i}}\left(8^{i}\right)^{\frac{1}{\gamma_i}}\|\hat{h}\|_{L^{\gamma_i}(Q^+_{r_i},x_n^{-\alpha})}.
\end{equation*}
Therefore, iteration yields
\begin{equation}\label{eq-lowerbound-of-min-overline-h}
\min_{Q^+_{\frac{1}{8}}}\hat{h}\geq C_{12}\|\hat{h}\|_{L^{-p_0}\left(B^+_{\left(\frac{1}{8}\right)^{\frac{1}{2-\alpha}}}\times\left(\frac{7}{8},1\right),x_n^{-\alpha}\right)}, \quad (C_{12}>0).
\end{equation}
To get the harnack's inequality for $\hat{h}$, we finally examine the case $\delta=-1$. Let $0<\rho\leq\left(\frac{1}{2}\right)^{\frac{1}{2-\alpha}}$. We take $\phi(x,t)=\eta^{2}(x,t)\hat{h}^{-1}(x,t)$ as test function, where $\eta\in C_0^{\infty}(B^+_1\times[0,1])$, $\eta(x,t)=1$ in $Q^+_{\rho}$, $\textbf{supp}(\eta)\subset Q^+_{2^{\frac{1}{2-\alpha}}\rho}\subset Q^+_{1}$ and $\nabla\eta\leq\frac{C_{13}}{\rho}$, $|\eta_t|\leq \frac{C_{13}}{\rho^{2-\alpha}}$ for some constant $C_{13}>0$. Then, since $w$ satisfies \eqref{eq-general-of-tilde-h-x-n-1} in $Q^+_1$, we can get
\begin{equation*}
\left(\frac{1}{\left|x_n^{-\alpha}(B^+_{\rho})\right|}\iint_{Q^+_{\rho}}|\nabla\log\hat{h}|^2\,dxdt\right)^{\frac{1}{2}}\leq C_{14}
\end{equation*}
for some constant $C_{14}>0$. Let $U=\log\hat{h}$ and $p=\frac{2(n-\alpha)}{n-2}$. Then, by the Lemma \ref{lem-weighted-sobolev-inequality-in-H-Koch-paper}, there exists constant $C_{15}$, $C_{16}>0$ such that
\begin{equation}\label{eq-aligned-for-applying-john-nerenberg-lemma-2}
\begin{aligned}
\frac{1}{\left|x_n^{-\alpha}(B^+_{\rho})\right|}\iint_{Q^+_{\rho}}x_n^{-\alpha}|U-U_{Q^+_{\rho}}|\,dxdt &\leq \left|x_n^{-\alpha}(B^{+}_{\rho})\right|^{-\frac{1}{p}}\cdot\int_{1-\rho^{2-\alpha}}^{1}\left(\int_{B^+_{\rho}}x_n^{-\alpha}|U-U_{Q^+_{\rho}}|^p\,dx\right)^{\frac{1}{p}}dt\\
&\leq C_{15}\rho^{-\frac{n-2}{2}}\int_{1-\rho^{2-\alpha}}^{1}\left(\int_{B^+_{\rho}}|\nabla U|^2\,dx\right)^{\frac{1}{2}}dt\\
&\leq C_{16}\rho^{-\frac{n-2}{2}+\frac{n-\alpha}{2}+2-\alpha}=C_{16}\rho^{3-\frac{3}{2}\alpha} \leq C_{16}.
\end{aligned}
\end{equation}
The parabolic version of John and Nirenberg Lemma for BMO (see \cite{Ai, FG}) yields that there exist two positive constants $P_0$ and $C_{17}$ such that
\begin{equation}\label{eq-applying-John-Nirenberg-Lemma}
\begin{aligned}
&\left(\frac{1}{\left|x_n^{-\alpha}\left(B^+_{\left(\frac{1}{8}\right)^{\frac{1}{2-\alpha}}}\times\left(\frac{1}{2},\frac{5}{8}\right)\right)\right|}\int_{\frac{1}{2}}^{\frac{5}{8}}\int_{B^+_{\left(\frac{1}{8}\right)^{\frac{1}{2-\alpha}}}}x_n^{-\alpha}e^{p_0 U}\,dxdt\right)^{\frac{1}{p_0}}\\
&\qquad \qquad \qquad \times\left(\frac{1}{\left|x_n^{-\alpha}\left(B^+_{\left(\frac{1}{8}\right)^{\frac{1}{2-\alpha}}}\times\left(\frac{7}{8},1\right)\right)\right|}\int_{\frac{7}{8}}^1\int_{B^+_{\left(\frac{1}{8}\right)^{\frac{1}{2-\alpha}}}}x_n^{-\alpha}e^{-p_0 U}\,dxdt\right)^{\frac{1}{p_0}}\leq C_{17}.
\end{aligned}
\end{equation} 
Then combining this with \eqref{eq-upperbound-of-max-overline-h} and \eqref{eq-lowerbound-of-min-overline-h}, we get
\begin{equation*}
\max_{Q^+_{\frac{1}{8}}\left(0,-\frac{3}{8}\right)}\hat{h}\leq C_{18}\min_{Q^+_{\frac{1}{8}}}\hat{h}
\end{equation*}
for some constant $C_{18}>0$ and lemma follows.
\end{proof}

We can now state the first main result of our paper.

\begin{thm}\label{theorem-Holder-estimates-for-h}
Let $w$ be a solution of equation \eqref{eq-general-of-tilde-h-x-n-1} defined in $B^+_1\times(0,1)$ with conditions \eqref{eq-assumption-for-coefficients-1} and \eqref{eq-assumption-for-coefficients-2}. Suppose that 
\begin{equation*}
\max_{Q^+_1}|w|\leq C_{|w|}<\infty.
\end{equation*}
Then, $w$ is locally H\"older continuous in $B_1^+\times(0,1)$ and 
\begin{equation*}
\|w\|_{C^{\gamma}(Q^+_s)}\leq C\left(\|w\|_{L^{\infty}(Q^+_1)}+1\right), \qquad (s<1).
\end{equation*}
\end{thm}

\begin{proof}
For $\overline{h}=w+C_{|w|}$, let $m_r=\inf_{Q^+_r}\overline{h}$, $M_r=\sup_{Q^+_r}\overline{h}$. Then $\overline{h}-m_r+1$, $M_r-\overline{h}+1\geq 1$ and satisfy the equation \eqref{eq-equaiton-for-hat-h-20} in $Q_r$. Applying the Harnack inequality (Lemma \ref{lem-Harnack-s-Inequality}) to those equations, we get
\begin{equation*}
\begin{aligned}
M_r-M_{\frac{r}{8}}+1=\inf_{Q^+_{\frac{r}{8}}}(M_r-\overline{h}+1)\geq \frac{1}{C}\sup_{Q^+_{\frac{r}{8}}\left(0,-\frac{5r}{8}\right)}(M_r-\overline{h}+1)=\frac{1}{C}\left(M_{r}-m_{\frac{r}{8}}+1\right)
\end{aligned}
\end{equation*}
and 
\begin{equation*}
\begin{aligned}
m_{\frac{r}{8}}-m_r+1=\inf_{Q^+_{\frac{r}{8}}}(\overline{h}-m_r+1)\geq \frac{1}{C}\sup_{Q^+_{\frac{r}{8}}\left(0,-\frac{5r}{8}\right)}(\overline{h}-m_r+1)=\frac{1}{C}\left(M_{\frac{r}{8}}-m_r+1\right).
\end{aligned}
\end{equation*}
Hence,
\begin{equation*}
M_{\frac{r}{8}}-m_{\frac{r}{8}}\leq \left(\frac{C-1}{C+1}\right)\left(M_r-m_r\right)+\frac{2(C-1)}{C+1}.
\end{equation*}
Let $\osc(r)=M_r-m_r$. Then
\begin{equation*}
\osc\left(\frac{r}{8}\right)\leq \left(\frac{C-1}{C+1}\right)\osc(r)+\frac{2(C-1)}{C+1}.
\end{equation*}
By an elementary iteration, we get
\begin{equation*}
\osc(r)\leq C'r^{\gamma}\left[\osc(1)+1\right]
\end{equation*}
for some constant $C'>0$ depending on $C$, i.e., $\overline{h}$ is H\"older at $(x,t)=(0,1)$. Since $\overline{h}=w+C_{|w|}$, we also have H\"older continuity of $w$ at $(x,t)=(0,1)$. Finally, by translating and dilating, we have
\begin{equation*}
\|w\|_{C^{\gamma}(Q^+_s)}\leq C''\left(\|w\|_{L^{\infty}(Q^+_1)}+1\right), \qquad (s<1)
\end{equation*}
and lemma follows.
\end{proof}

\section{H\"older estimates II}
\setcounter{equation}{0}
\setcounter{thm}{0}
Through the previous section, we have established that the gradient of $h$ in \eqref{eq-for-sol-h-of-fixed-bondary-problem-in-intro} with respect to $x_i$, $(i=1,\cdots,n-1)$, is H\"older continuous near the boundary $\partial\Omega$. However, one can easily check that $h_{z}$ is not a solution of \eqref{eq-for-sol-h-x-i-of-fixed-bondary-problem-in-intro}. Hence, solving the missing step \eqref{eq-missing step for the regularity} is not completed at this moment. Thus, we will devote this section for the H\"older regularity for $h_z$.\\
\indent Let $w$ be a solution of the problem
\begin{equation}\label{eq-general-of-tilde-h-x-n-1-original}
\begin{cases}
\begin{aligned}
&w_{t}=x_n^{\alpha}\left(a^{ij} w_{ij}\right) \qquad \qquad \mbox{in $H$}\\
&w(x,t)=0 \qquad \qquad \qquad \mbox{on $x_n=0$}\\
&w(x,0)=w_0(x) \qquad \qquad \quad \mbox{in $H$}
\end{aligned}
\end{cases}
\end{equation}
with initial value $w_0(x)\in C^{0,1}(H)$. Assume that the coefficients $a^{ij}(x,t)$ are measurable functions and satisfy
\begin{equation}\label{eq-assumption-for-coefficients-1-for-original}
\lambda|\xi|^2\leq a^{ij}(x,t)\xi_i\xi_j\leq \Lambda|\xi|^2, \qquad (i,j=1,\cdots,n)
\end{equation}
for some constants $0<\lambda\leq\Lambda<\infty$. \\
\indent To get the H\"older regularity for $h_z$, we will first construct an important, for our purpose, barrier function.

\begin{lem}\label{lem-barrier-fuction-for-subsolution-with-nontrivial-gradient}
Let $\mathcal{Q}=B^+_r\times(0,1)$ and $P_r(x')=(x',r)\in H$ for $0<r<1$ and let $K=\left\{x\in B_{r}^+:|x_i-\left(P_{\frac{r}{4}}(0)\right)_i|<\frac{r}{8}, \,\,i=1,\cdots,n\right\}$. Then, there exists a solution $f>0$ and time $T(r)>0$ such that
\begin{equation}\label{eq-cases-aligned-barrier-function-f}
\begin{cases}
\begin{aligned}
x_n^{\alpha}\left(a^{ij}f_{ij}\right)&-f_t=0 \qquad \qquad \qquad  (x,t)\in \mathcal{Q}\\
f&=0 \qquad \qquad \qquad \qquad \partial_l\mathcal{Q}\\
|\nabla f|&\geq c(T)>0 \qquad \qquad \partial B^+_{\frac{r}{16}}\cap\{x_n=0\}, \quad t\geq T.
\end{aligned}
\end{cases}
\end{equation}
Moreover, at $t=0$, $f$ satisfies
\begin{equation}\label{eq-cases-aligned-barrier-function-f-initial-condition}
\supp\{f(\cdot,0)\}\subset K \qquad \mbox{and} \qquad 0\leq f(x,0) \leq m_0\textbf{1}_{_{K}}<\infty
\end{equation}
for a constant $m_0>0$.
\end{lem}

\begin{proof}
In this proof, we will use a modification of the technique of \cite{KL2} to prove the lemma. For each $x'\in \R^{n-1}$ such that $|x'|<\frac{r}{16}$, let $P_{\frac{r}{4}}=P_{\frac{r}{4}}(x')$ and consider
\begin{equation*}
g(x,t)=\frac{1}{(4\pi t)^{\frac{n}{2}}}e^{-\frac{\beta\left|x-P_{\frac{r}{4}}\right|^2}{t}}
\end{equation*}
and $\tilde{g}(x,t)=e^{-Mt}g(x,t+\tau_0)$ for $\tau_0>0$ we can choose later. By a direct computation, we can get
\begin{equation*}
\tilde{g}_t(x,t)=e^{-Mt}g(x,t+\tau_0)\left[-M-\frac{n}{2(t+\tau_0)}+\frac{\beta\left|x-P_{\frac{r}{4}}\right|^2}{(t+\tau_0)^2}\right]
\end{equation*}
and
\begin{equation*}
\tilde{g}_{ij}=e^{-Mt}g(x,t+\tau_0)\left[\frac{4\beta^2\left(x_i-P_{\frac{r}{4},i}\right)\left(x_j-P_{\frac{r}{4},j}\right)}{(t+\tau_0)^2}\right], \qquad (i\neq j \,\, \mbox{and}\,\, 1\leq i,j\leq n)
\end{equation*}
and 
\begin{equation*}
\tilde{g}_{ii}=e^{-Mt}g(x,t+\tau_0)\left[-\frac{2\beta}{t+\tau_0}+\frac{4\beta^2\left(x_i-P_{\frac{r}{4},i}\right)^2}{(t+\tau_0)^2}\right], \qquad (1\leq i\leq n).
\end{equation*}
Hence
\begin{equation*}
x_n^{\alpha}\left(a^{ij}\tilde{g}_{ij}\right)-\tilde{g}_t\geq e^{-Mt}g(x,t+\tau_0)\left[M-\frac{2\beta\Lambda r^{\alpha}}{\tau_0}-\frac{4\beta r^2}{\tau_0^2}\right]\geq 0
\end{equation*}
if 
\begin{equation}\label{eq-range-of-M-lower-bound}
M\geq \frac{2\beta\Lambda r^{\alpha}}{\tau_0}+\frac{4\beta r^2}{\tau_0^2}.
\end{equation}
Now, we define the function $\tilde{g}_{_{\epsilon_0}}$ by
\begin{equation*}
\tilde{g}_{_{\epsilon_0}}=\max\{0, \tilde{g}-\epsilon_0\}
\end{equation*}
for small $\epsilon_0>0$. Since $\tilde{g}_{_{\epsilon_0}}$ takes maximum value of two subsolutions to \eqref{eq-general-of-tilde-h-x-n-1-original} at each point in $H$, it is also a subsolution of \eqref{eq-general-of-tilde-h-x-n-1-original}. We are going to choose proper constants $\beta$, $\tau_0$ and $\epsilon_0$ so that
\begin{equation}\label{eq-condition-1-for-barrier-subsolution}
\supp\{\tilde{g}_{_{\epsilon_0}}(\cdot,0)\}\subset K \qquad \mbox{and} \qquad 0<\tilde{g}_{_{\epsilon_0}}(P_r,0)\leq m_0
\end{equation}
and there is a constant $T>0$ satisfying
\begin{equation}\label{eq-condition-2-for-barrier-subsolution}
\supp\{\tilde{g}_{_{\epsilon_0}}(\cdot,T)\}=B_{\frac{r}{4}}(P_{\frac{r}{4}}).
\end{equation}
\indent For \eqref{eq-condition-1-for-barrier-subsolution}, we take sufficiently small constant $\tau_0$ and large one $\beta>>1$ so that
\begin{equation}\label{eq-choose-constant-tau-0-and-beta-larger-than-1-1}
\frac{1}{\left(4\pi\tau_0\right)^{\frac{n}{2}}}e^{-\frac{\beta r^2}{16^2\tau_0}}<\epsilon_0 \qquad \mbox{and} \qquad \epsilon_0<\frac{1}{\left(4\pi\tau_0\right)^{\frac{n}{2}}}<m_0+\epsilon_0.
\end{equation}
Then, 
\begin{equation*}
\tilde{g}(\cdot,0)<\epsilon_0 \quad \mbox{on $B_r^+\bs K$} \qquad \mbox{and} \qquad \epsilon_0<\tilde{g}(P_r,0)\leq m_0+\epsilon_0.
\end{equation*}
Thus, the condition \eqref{eq-condition-1-for-barrier-subsolution} holds for sufficiently large constant $\tau_0$ and $\beta$.\\
\indent For \eqref{eq-condition-2-for-barrier-subsolution}, we will focus on $\supp\{\tilde{g}_{_{\epsilon_0}}(\cdot,t)\}$. Since $\tilde{g}_{_{\epsilon_0}}(x,t)$ is radially symmetric about the point $P_{\frac{r}{2}}$, for each $t\geq 0$, $\epsilon_0$-level set of $\tilde{g}_{_{\epsilon_0}}(\cdot,t)$ is a ball centered at $P_{\frac{r}{2}}$. Let's denote by $R(t)$ the radius of the $\epsilon_0$-level set of $\tilde{g}_{_{\epsilon_0}}(\cdot,t)$. Then, one can easily check that there exists time $t=t'$ such that $R(t)$ is increasing on $(0,t')$ and decreasing on $(t',\infty)$. Hence, at time $t'>0$, $R(t)$ satisfies 
\begin{equation*}
\frac{d \{R(t)\}}{dt}=0 \qquad \mbox{at $t=t'$}.
\end{equation*}
To show that there exists a time $T>0$ such that 
\begin{equation*}
\supp\{\tilde{g}_{_{\epsilon_0}}(\cdot,T)\}=B_{\frac{r}{4}}(P_{\frac{r}{4}}),
\end{equation*}
we will show that $R(t')\geq \frac{r}{2}$ with the following equation
\begin{equation}\label{eq-epsilon-0-level-set-with-radius-R-t-29349}
e^{-Mt-\frac{\beta R(t)^2}{t+\tau_0}}=\epsilon_0[4\pi(t+\tau_0)]^{\frac{n}{2}}.
\end{equation}
Differentiating on both sides of \eqref{eq-epsilon-0-level-set-with-radius-R-t-29349} with respect to $t$, we can get
\begin{equation}\label{eq-lower-bound-of-R-t-with-M-and-tau-0-and-beta}
R(t')^2=\frac{(t'+\tau_0)^2}{\beta}\left[M+\frac{\epsilon_0n}{2(t'+\tau_0))}\right]\geq \frac{\tau_0^2M}{\beta} \qquad \mbox{at $t=t'$}.
\end{equation}
By \eqref{eq-range-of-M-lower-bound}, the last term of \eqref{eq-lower-bound-of-R-t-with-M-and-tau-0-and-beta} is bounded below by $4r^2$. Hence, we can get
\begin{equation*}
R(t')\geq 2r.
\end{equation*}
This immediately implies that there exists a time $0<T\leq t'$ such that
\begin{equation*}
R(T)=\frac{r}{4}
\end{equation*}
and \eqref{eq-condition-2-for-barrier-subsolution} holds. Since \eqref{eq-condition-2-for-barrier-subsolution} is true, one can easily check that there exists a constant $c=c(T)$ such that
\begin{equation*}
\left|\nabla\tilde{g}_{\epsilon_0}(P_0,T)\right|\geq c(T)>0.
\end{equation*}
\noindent Next, we define function $G$ by
\begin{equation*}
\begin{aligned}
G(x,t)&=\epsilon_1\left(1-\frac{4\left|x-P_{\frac{r}{4}}\right|^2}{r^2}\right)\tilde{g}(x,t)\\
&=\frac{\epsilon_1}{(4\pi(t+\tau_0))^{\frac{n}{2}}}\left(1-\frac{4\left|x-P_{\frac{r}{4}}\right|^2}{r^2}\right)e^{-Mt-\frac{\beta\left|x-P_{\frac{r}{4}}\right|^2}{t+\tau_0}} \quad \mbox{in $B_{\frac{r}{4}}(P_{\frac{r}{4}})\times[T,1)$}.
\end{aligned}
\end{equation*}
for some constant $\epsilon_1>0$. Note that if we choose $\epsilon_1$ sufficiently small, then 
\begin{equation*}
G(x,T)\leq \tilde{g}_{\epsilon}(x,T), \qquad \mbox{in $B_{\frac{r}{4}}(P_{\frac{r}{4}})$}
\end{equation*}
since, for some constants $c_1$ and $c_2$,
\begin{equation*}
|\nabla\tilde{g}_{\epsilon}(\cdot,T)|\geq c_1>0 \quad \mbox{and} \quad |\nabla G|\leq c_2<\infty \qquad \mbox{on $\partial B_{\frac{r}{4}}(P_{\frac{r}{4}})$}.
\end{equation*}
On the other hand, by a direct computation, we can get
\begin{equation*}
G_t(x,t)=G\left[-M-\frac{n}{2(t+\tau_0)}+\frac{\beta\left|x-P_{\frac{r}{4}}\right|^2}{(t+\tau_0)^2}\right]
\end{equation*}
and
\begin{equation*}
\begin{aligned}
G_{ij}&=G\left[\frac{4\beta^2\left(x_i-P_{\frac{r}{4},i}\right)\left(x_j-P_{\frac{r}{4},j}\right)}{(t+\tau_0)^2}\right]\\
&\qquad \quad +\epsilon_1\tilde{g}\left[\frac{8\beta\left(x_i-P_{\frac{r}{4},i}\right)\left(x_j-P_{\frac{r}{4},j}\right)}{r^2(t+\tau_0)}\right], \qquad (i\neq j \,\, \mbox{and}\,\, 1\leq i,j\leq n)
\end{aligned}
\end{equation*}
and 
\begin{equation*}
G_{ii}=G\left[-\frac{2\beta}{t+\tau_0}+\frac{4\beta^2\left(x_i-P_{\frac{r}{4},i}\right)^2}{(t+\tau_0)^2}\right]+\epsilon_1\tilde{g}\left[\frac{8\beta|x_i-P_{\frac{r}{4},i}|^2}{r^2(t+\tau_0)}-\frac{8}{r^2}\right], \qquad (1\leq i\leq n).
\end{equation*}
Then
\begin{equation}\label{eq-aligned-show-G-to-be-sub-solution}
\begin{aligned}
x_n^{\alpha}\left(a^{ij}G_{ij}\right)-G_t\geq & G\left[\frac{M}{2}-\frac{2\beta r^{\alpha}}{T+\tau_0}-\frac{r^2}{16(T+\tau_0)^2}\right]\\
&+\epsilon_1\tilde{g}\left[\frac{M}{2}\left(1-\frac{4\left|x-P_{\frac{r}{4}}\right|^2}{r^2}\right)+\frac{8\lambda\beta x_n^{\alpha}\left|x-P_{\frac{r}{4}}\right|^2}{r^2(1+\tau_0)}-\frac{8\Lambda x_n^{\alpha}}{r^2}\right]
\end{aligned}
\end{equation}
in $B_{\frac{r}{4}}(P_{\frac{r}{4}})\times[T,1)$. Let
\begin{equation}\label{eq-range-of-beta-and-M-lower-bound-2}
\beta\geq \frac{64\Lambda(1+\tau_0)}{\lambda r^2} \qquad \mbox{and} \qquad M\geq \frac{256\Lambda}{15r^2}.
\end{equation}
Then, the second term on the right hand side of \eqref{eq-aligned-show-G-to-be-sub-solution} is positive. Hence, 
\begin{equation*}
x_n^{\alpha}\left(a^{ij}G_{ij}\right)-G_t\geq 0 \qquad \mbox{in $B_{\frac{r}{4}}(P_{\frac{r}{4}})\times[T,1)$}
\end{equation*}
if
\begin{equation}\label{eq-range-of-M-lower-bound-2}
M\geq \max\left\{\frac{4\beta\Lambda r^{\alpha}}{\tau_0}+\frac{8\beta r^2}{\tau_0^2},\frac{256\Lambda}{15r^2}\right\}.
\end{equation}
\indent Now we let $f$ be a solution of 
\begin{equation}\label{eq-cases-aligned-barrier-function-f-approximated-by-f}
\begin{cases}
\begin{aligned}
x_n^{\alpha}\left(a^{ij}f_{ij}\right)&-f_t=0 \qquad \qquad \qquad  (x,t)\in \mathcal{Q}\\
f(x,t)&=0 \qquad \qquad \qquad \qquad (x,t)\in\partial_l\mathcal{Q}\\
f(x,0)=&m_0\textbf{1}_{K} \qquad \qquad \qquad x\in B_{r}^{+}.
\end{aligned}
\end{cases}
\end{equation}
Then, for sufficiently small $\epsilon_1>0$ and $M$ in \eqref{eq-range-of-M-lower-bound-2} and $\beta$ satisfying \eqref{eq-choose-constant-tau-0-and-beta-larger-than-1-1} and \eqref{eq-range-of-beta-and-M-lower-bound-2}, we have
\begin{equation*}
f\geq \tilde{g}_{\epsilon} \qquad \mbox{in $B_r^+\times(0,T]$}
\end{equation*}
and 
\begin{equation*}
f\geq G \qquad \mbox{in $B_{\frac{r}{4}}(P_{\frac{r}{4}})\times[T,1)$}
\end{equation*}
by the comparison principle. Since $x'\in\R^{n-1}$ is an arbitrary point with $|x'|\leq\frac{r}{16}$, by the definition of functions $\tilde{g}_{\epsilon_0}$ and $G$, the solution $f$ has nontrivial gradients on $\partial B_{\frac{r}{16}}^+\cap\{x_n=0\}$, i.e., the third inequality of \eqref{eq-cases-aligned-barrier-function-f} holds.\\ 
\indent To complete the proof of lemma, we are now going to show the existence of solution $f$ to \eqref{eq-cases-aligned-barrier-function-f} with initial condition \eqref{eq-cases-aligned-barrier-function-f-initial-condition}. We begin by constructing a sequence of approximate domian so as to avoid the degeneracy of the equation. We may simply put
\begin{equation*}
B^+_{r,n}=B_{r}^+\cap\left\{x_n>\frac{r}{n}\right\}.
\end{equation*}
We now solve the problem
\begin{equation}\label{eq-cases-aligned-barrier-function-f-approximated-by-f-n}
\begin{cases}
\begin{aligned}
x_n^{\alpha}\left(a^{ij}(f_n)_{ij}\right)&-(f_n)_t=0 \qquad \qquad \qquad  (x,t)\in B_{r,n}^+\times(0,1)\\
f_n(x,t)&=0 \qquad \qquad \qquad \qquad (x,t)\in\partial_lB_{r,n}^+\times(0,1)\\
f_n(x,0)=&m_0\textbf{1}_{K} \qquad \qquad \qquad x\in B_{r,n}^+.
\end{aligned}
\end{cases}
\end{equation}
The maximum principle, which holds for classical solutions, implies that
\begin{equation*}
0\leq f_n\leq m_0 \qquad \mbox{in $B_{r,n}^+\times(0,1)$}.
\end{equation*}
Moreover, again by the maximum principle
\begin{equation*}
f_{n}\leq f_{n+1} \qquad \mbox{in $B_{r,n}^+\times(0,1)$}, \qquad \forall n\geq 4.
\end{equation*}
Hence we can define the function
\begin{equation*}
f(x,t)=\lim_{n\to\infty}f_n(x,t), \qquad \mbox{in $\mathcal{Q}$} 
\end{equation*}
as a monotone limit of bounded non-negative functions. On the other hand, we also have
\begin{equation*}
f_n(x,0)\leq \frac{4m_0}{r_0}x_n \qquad \mbox{in $B_{r,n}^+$}.
\end{equation*}
Hence, $0\leq f\leq Mx_n$ for some constant $M>0$ and $f$ is continuous up to the boundary. Since $f_n$ is a classical solution of \eqref{eq-cases-aligned-barrier-function-f-approximated-by-f-n}, $f$ clearly satisfies \eqref{eq-cases-aligned-barrier-function-f} and \eqref{eq-cases-aligned-barrier-function-f-initial-condition}. Hence, we complete the proof.
\end{proof}

With this barrier function, we are going to show the second main result for our paper: $C^{1,\gamma}$ regularity of solution $w$ to \eqref{eq-general-of-tilde-h-x-n-1-original}.

\begin{lem}\label{lem-C-1-gamma-regularity-near-the-boundary}
Let $w$ be a solution of \eqref{eq-general-of-tilde-h-x-n-1-original}. If $w$ is in $C^{0,1}(H)$, with respect to space variable $x$, then, it is also in $C^{1,\gamma}$, $(0<\gamma<1)$ near the plane $\{x_n=0\}$, i.e., for each $x_0\in \{x_n=0\}$ and $T>0$, there is a constant $L$ which is depending on $x_0$ such that
\begin{equation}\label{eq-in-lemma-for-c-1-beta-regularity}
|w(x,T)-Lx_n|\leq C_0|x-x_0|^{1+\gamma}, \qquad \mbox{in $B_r^+(x_0)$}
\end{equation}
for some constant $C_0>0$
\end{lem}

\begin{proof}
For each $T>0$, let $r=r(lT)$ be given in Lemma \ref{lem-barrier-fuction-for-subsolution-with-nontrivial-gradient} for $l=1-(16)^{\alpha-2}$. We denote by 
\begin{equation*}
M_0=\max_{K}w_0 \qquad \mbox{and} \qquad m_0=\min_{K}w_0
\end{equation*}
for the set $K=\left\{x\in B_{r}^+:|x_i-P_{\frac{r}{4},i}|<\frac{r}{8}, \,\,P_{\frac{r}{4}}=(0,\cdots,0,\frac{r}{4})\in\R^n\,\, \mbox{and}\,\,i=1,\cdots,n\right\}$. Then, by the Lemma \ref{lem-barrier-fuction-for-subsolution-with-nontrivial-gradient}, there exists a solution $f$ of
\begin{equation*}
\begin{cases}
\begin{aligned}
f_t&=x_n^{\alpha}\left(a^{ij}f_{ij}\right) \qquad \mbox{in $B_{r}^+$}\\
f&=0 \qquad \qquad \quad \mbox{on $\partial B_{r}^{+}$}
\end{aligned}
\end{cases}
\end{equation*}
with the conditions $w(x,0)\geq f(x,0)$ in $B^+_{r}$ and
\begin{equation*}
|\nabla f(x,t)|\geq c_0>0, \quad \forall t\geq lT, \,\, x\in\partial B^+_{\frac{r}{16}}\cap\{x_n=0\}.
\end{equation*}
Hence,
\begin{equation*}
|\nabla w(x,t)|\geq |\nabla f(x,t)|\geq c_0>0, \quad \forall t\geq lT, \,\, x\in\partial B^+_{\frac{r}{16}}\cap\{x_n=0\}
\end{equation*}
and this implies that there exists a constant $A>0$ such that
\begin{equation}\label{eq-bounded-below-of-w-by-A-x-n}
w(x,t)\geq Ax_n \qquad \mbox{in $B^+_{\frac{r}{16}}\times\left[lT,T\right]$}.
\end{equation}
In addition, there also exists a constant $B>0$ such that
\begin{equation}\label{eq-bounded-above-of-w-by-B-x-n}
w(x,t)\leq Bx_n \qquad \mbox{in $B^+_{\frac{r}{16}}\times\left[lT,T\right]$}
\end{equation}
since $w_0\in C^{0,1}(H)$ and the function $Bx_n$ is a solution of \eqref{eq-general-of-tilde-h-x-n-1-original}. Now we are going to show \eqref{eq-in-lemma-for-c-1-beta-regularity} with scaling properties. To get them, we define the new function by
\begin{equation*}
w'(x',t')=16w(x,t) \qquad \left(x'=16x,\quad t'=(16)^{2-\alpha}\left(t-lT\right)\right).
\end{equation*}
Then, the function $w'$ is also a solution of \eqref{eq-general-of-tilde-h-x-n-1-original} with initial data $w(\cdot,lT)$. By \eqref{eq-bounded-below-of-w-by-A-x-n} and \eqref{eq-bounded-above-of-w-by-B-x-n}, we also obtain
\begin{equation}\label{eq-trapping-w'-with-A-and-B-x-n}
Ax_n'\leq w'(x',t') \leq Bx_n', \qquad \mbox{in $B_r^+\times[0,T]$}.
\end{equation}
Let's assume that $A_0$ and $B_0$ are the optimal constants for \eqref{eq-trapping-w'-with-A-and-B-x-n}. If 
\begin{equation*}
\left|\left\{x'\in B_r^+:w'(x',0)\geq \left(\frac{A_0+B_0}{2}\right)x_n'\right\}\right|\geq \frac{1}{2}\left|B_{r}^+\right|,
\end{equation*}
then there exist constants $\delta_1>0$ and $m_1>0$ such that
\begin{equation}\label{eq-w--A-x-n-bound-below-by-m-1-1}
w'(x',0)-A_0x_n'\geq m_1, \qquad \mbox{in $K$}.
\end{equation}
This is because $w'(x',0)-A_0x_n$ is Lipschitz continuous with respect to space variable. On the other hand, if
\begin{equation*}
\left|\left\{x'\in B_r^+:w'(x',0)\leq \left(\frac{A_0+B_0}{2}\right)x_n'\right\}\right|\geq \frac{1}{2}\left|B_{r}^+\right|,
\end{equation*}
then there exist constants $\delta'_1>0$ and $m'_1>0$ such that
\begin{equation*}
B_0x_n'-w'(x',0)\geq m'_1, \qquad \mbox{in $K$}.
\end{equation*}
Note that the constants $m_1$ and $m_1'$ are only obtained by $|A_0-B_0|$ and Lipschitz continuity. We first assume that
\begin{equation*}
\left|\left\{x'\in B_r^+:w'(x',0)\geq \left(\frac{A_0+B_0}{2}\right)x_n'\right\}\right|\geq \frac{1}{2}\left|B_{r}^+\right|.
\end{equation*}
Then, by similar argument as above, there is a function $f'$ such that
\begin{equation*}
w'(x',t')-A_0x_n'\geq f'(x',t') \quad \mbox{in $B^+_{r}\times(0,T)$} 
\end{equation*}
and
\begin{equation*} 
|\nabla f'(x',t')|\geq c_0>0, \quad \forall t'\geq lT, \,\, x'\in\partial B^+_{\frac{r}{16}}\cap\{x'_n=0\}.
\end{equation*}
Note that $m_1$ in \eqref{eq-w--A-x-n-bound-below-by-m-1-1} is only depending on $|B-A|$ and Lipschitz continuity of $w'$. Hence, the constant $c_0$ is proportion to $|A_0-B_0|$. Thus, there exists a constant $0<\kappa<\frac{15}{16}$ such that
\begin{equation*}
w'(x',t')-Ax_n'\geq \kappa(B_0-A_0)x'_n\qquad \mbox{in $B^+_{\frac{r}{16}}\times[lT,T]$}.
\end{equation*}
Similarly, if 
\begin{equation*}
\left|\left\{x'\in B_r^+:w'(x',0)\leq \left(\frac{A_0+B_0}{2}\right)x_n'\right\}\right|\geq \frac{1}{2}\left|B_{r}^+\right|,
\end{equation*}
then, we also have
\begin{equation*}
B_0x_n'-w'(x',t')\geq \kappa(B_0-A_0)x'_n\qquad \mbox{in $B^+_{\frac{r}{16}}\times[lT,T]$}.
\end{equation*}
Therefore, we get $B_1\geq A_1>0$ such that $0<A_0\leq A_1\leq B_1\leq B_0<\infty$, $B_1-A_1\leq (1-\kappa)(B_0-A_0)$ and 
\begin{equation}\label{eq-trapping-w'-with-A-and-B-x-n-2}
A_1x_n'\leq w'(x',t') \leq B_1x_n', \qquad \mbox{in $B_{\frac{r}{16}}^+\times[lT,T]$}.
\end{equation}
Hence, from \eqref{eq-trapping-w'-with-A-and-B-x-n-2},
\begin{equation*}
A_1x_n\leq w(x,t) \leq B_1x_n, \qquad \mbox{in $B_{\frac{r}{16^2}}^+\times[lT+(1-l)lT,T]$}.
\end{equation*}
By iteration arguments, we conclude that there exist series $\{A_k\}$ and $\{B_k\}$ such that $0<A_0\leq A_1\leq \cdots \leq A_k\leq A_{k+1}\leq \cdots\leq B_{k+1}\leq B_{k}\leq \cdots\leq B_1\leq B_0<\infty$, $B_{k+1}-A_{k+1}\leq (1-\kappa)(B_k-A_k)$ and 
\begin{equation*}
A_kx_n\leq w(x,t) \leq B_kx_n, \qquad \mbox{in $B_{\frac{r}{16^{k+1}}}^+\times[(1-(1-l)^{k+1})T,T]$}.
\end{equation*}
This will imply
\begin{equation}\label{eq-existence-linear-functional-Cz-corresponding-to-w}
\frac{|w(x,t)-Cx_n|}{x_n}\leq B_k-A_k \leq (1-\kappa)^{k}(B_0-A_0) , \qquad \mbox{in $B_{\frac{r}{16^{k+1}}}^+\times[(1-(1-l)^{k+1})T,T]$}
\end{equation}
for $C=\lim_{k\to\infty}A_k=\lim_{k\to\infty}B_k$. Now we get \eqref{eq-existence-linear-functional-Cz-corresponding-to-w} when $x\in B_{\frac{r}{16^{k+1}}}^+$. Hence, by the the property of logarithm function, we get
\begin{equation*}
\frac{|w(x,T)-Cx_n|}{x_n}\leq \left(\frac{16|x|}{r}\right)^{-\log_{16}(1-\kappa)}, \qquad \mbox{in $B^+_r$}.
\end{equation*}
Since $0<\kappa<\frac{15}{16}$, we have $0<-\log_{16}(1-\kappa)<1$. Hence \eqref{eq-in-lemma-for-c-1-beta-regularity} holds for constant $\gamma=-\log(1-\kappa)$ and the lemma follows.
\end{proof}

From the previous Lemma \ref{lem-C-1-gamma-regularity-near-the-boundary}, we can cover the remaining part of missing step for the long time existence.

\begin{cor}\label{cor-for-holder-estimate-for-h-z}
Under the hypotheses of Lemma \ref{lem-C-1-gamma-regularity-near-the-boundary}, we also have
\begin{equation*}
|w_z(x)-L|\leq C_1|x-x_0|^{\gamma}, \qquad \mbox{in $B_r^{+}(x_0)$}
\end{equation*}
for some constant $C_1>0$. 
\end{cor}

\begin{proof}
Let $x_0$ be a point on $\{x_n=0\}$ and let $x^s=(x',s)$ be a point in $B_r^+(x_0)$. Then,
\begin{equation}\label{eq-expression-of-w-sub-z-at-x-a}
w_z(x^a)=\lim_{b\to a}\frac{w(x^b)-w(x^a)}{b-a}.
\end{equation}
By the previous Lemma (Lemma \ref{lem-C-1-gamma-regularity-near-the-boundary}), we can get
\begin{equation}\label{eq-taylor-expansion-of-w-L-a}
|w(x^a)-La|=C_1|x^a-x_0|^{1+\gamma}+\mbox{\textbf{h.o.t.} of $|x^a-x_0|$}
\end{equation}
and
\begin{equation}\label{eq-taylor-expansion-of-w-L-b}
|w(x^b)-Lb|=C_1|x^b-x_0|^{1+\gamma}+\mbox{\textbf{h.o.t.} of $|x^b-x_0|$}
\end{equation}
for some constant $C_1>0$ if $x^a$ is sufficiently close to $x^b$. Here, \textbf{h.o.t.} means 'higher order terms'. Then, by \eqref{eq-expression-of-w-sub-z-at-x-a}, \eqref{eq-taylor-expansion-of-w-L-a} and \eqref{eq-taylor-expansion-of-w-L-b}, we obtain
\begin{equation*}
|w_z(x^a)-L|\leq \left|\lim_{b\to a}\frac{w(x^b)-Lb-(w(x^a)-La)}{b-a}\right|\leq C_1|x^a-x_0|^{\gamma}+\mbox{\textbf{h.o.t.} of $|x^a-x_0|$}.
\end{equation*}
Since $x^a$ is chosen arbitrary in $B_{r}^+(x_0)$, we can get a desired result.
\end{proof}

\section{Proof of Main Theorem}\label{section-final}
\setcounter{equation}{0}
\setcounter{thm}{0}

\subsection{Local coordinate change}
To motivate the proof of Main Theorem, we will first compute the
transformation of the equation
\begin{equation*}
f_t=mf^{\alpha}\La f
\end{equation*}
when one exchanges dependent and independent variables near the
boundary. This change of coordinates converts the boundary into a
flat boundary. More precisely, assume that the function $f$ belong
to the space $C_s^{2+\alpha}(\overline{\Omega}\times[0,T])$. Let
$X_0=(x^0,t_0)=(x'^0,x_n^0,t_0)$ at the boundary $\partial \overline{\Omega} \times
(0,T]$. We can assume (by rotating the coordinates) that
$f_{x_1}(X_0)=\cdots=f_{x_{n-1}}(X_0)=0$, and $f_{x_n}(X_0)>0$. It
follows from the Implicit Function Theorem that if the number $\eta$
is sufficiently small, we can solve the equation $z=f(x,t)$ with
respect to $x_n$ around the point $X_0$. This yields to a map
$x_n=h(x',z,t)$ defined in a small box $B^+_{\eta}(x'^0,0,t_0)$. We wish to
compute the evolution equation of $h$ from the one of $f$. To
compute the evolution of $h$, we use the identities
\begin{equation*}
\begin{array}{c}
f_{x_n}=\frac{1}{h_z},\quad f_{x_i}=-\frac{h_{x_i}}{h_z}, \quad
f_t=-\frac{h_t}{h_z},\quad f_{x_nx_n}=-\frac{1}{h_z^3}h_{zz},\\
f_{x_ix_i}=-\frac{1}{h_z}\left(\frac{h_{x_i}^2}{h_z^2}h_{zz}-2\frac{h_{x_i}}{h_z}h_{x_iz}+h_{x_ix_i}\right)
\quad(i=1,\cdots,n-1).
\end{array}
\end{equation*}
This follows immediately from above computations, the equation
\begin{equation*}
f_t-m{f}^{\alpha}\La f=0
\end{equation*}
transforms into the equation
\begin{equation}\label{eq-for-sol-h-of-fixed-bondary-problem}
M(h)=h_t-z^{\alpha}\Bigg[{\La}_{x'}h+\Bigg(-\frac{1+|\nabla_{x'}h|^2}{h_{z}}\Bigg)_{z}\Bigg]=0.
\end{equation}
Set
\begin{equation*}
a^{nn}=\frac{1+|\nabla_{x'}h|^2}{{h_z}^2}, \qquad a^{nk'}=-\frac{2h_{k'}}{h_z}\quad \mbox{and} \quad a^{k'l'}=\delta_{k'l'}
\end{equation*}
for $k',l'=1,\cdots,n-1$. Then, $h$ is a solution of the equation
\begin{equation*}
L[w]=w_t-z^{\alpha}\left(a^{kl}w_{kl}\right), \qquad (k,l=1,\cdots,n).
\end{equation*}
The operator $M$ defined above becomes degenerate when $z=0$. We can also easily compute its first derivatives with respect to $x_i$, $(i=1,\cdots,n-1)$:
\begin{equation}\label{eq-aligned-for-h-x-i}
\begin{aligned}
L[w:x_i]&=w_t-z^{\alpha}\left(\left(\frac{1+|\nabla_{x'}h|^2}{{h_z}^2}\right)w_z-\sum_{j=1}^{n-1}\frac{2h_{x_j}}{h_z}w_{x_j}\right)_z+z^{\alpha}\D_{x'}\cdot(\D_{x'}w)
\end{aligned}
\end{equation}
Set
\begin{equation*}
a^{nn}=\frac{1+|\nabla_{x'}h|^2}{{h_z}^2}, \qquad a^{k'n}=-\frac{2h_{k'}}{h_z}, \qquad a^{nk'}=0 \quad \mbox{and} \quad a^{k'l'}=\delta_{k'l'}
\end{equation*}
for $k',l'=1,\cdots,n-1$. Then, $h_{x_i}$ is a solution of the equation
\begin{equation*}
L[w:x_i]=w_t-z^{\alpha}\nabla_{l}\left(a^{kl}\nabla_{k}w\right), \qquad (k,l=1,\cdots,n).
\end{equation*}

We finish with the proof of Theorem \ref{The-Main-Theorem-of-this-paper}.

\begin{proof}[\textbf{The Proof of Theorem \ref{The-Main-Theorem-of-this-paper}}]
By the Theorem \ref{thm-short}, there exists a solution $f$ of \eqref{eq-cases-main-of-this-paper-section-PMT} which is $C_s^{2,\overline{\gamma}}$-smooth up to the boundary, when $0<t<\tau$ for some constant $\tau>0$. Let $(0,\tau_0)$ be the maximal interval of existence for $C_s^{2,\overline{\gamma}}$-solution $f$ of \eqref{eq-cases-main-of-this-paper-section-PMT} and suppose that $\tau_0<\infty$. By scaling time around $t=\tau_0$, we can assume that
\begin{equation*}
\tau_0=1.
\end{equation*}
We will show the solution $f$ is in $C_s^{2,\overline{\gamma}}(\Omega)$ at $t=1$. Then, by applying Theorem \ref{thm-short} again, we can extend $C_s^{2,\overline{\gamma}}$-smoothness of $f$ to $t=1+\tau_1$ for some constant $\tau_1>0$. By maximality of $\tau_0=1$, contradiction arise. Hence, $\tau_0=\infty$ and theorem follows.\\ 
\indent Since $\Omega$ is compact domain, we can express $\Omega$ as the finite union
\begin{equation}\label{eq-union-of-Omega-by-finite-compact-sets}
\Omega=\Omega_0\cup\left(\bigcup_{l\geq 1}\Omega_l\right)
\end{equation}
of compact domains in such a way that
\begin{equation*}
\textbf{dist}(\Omega_0,\partial\Omega)\geq \frac{\rho}{2}>0
\end{equation*}
and for all $l\geq 1$
\begin{equation*}
\Omega_l=B_{\rho}(x_l)\cap\Omega
\end{equation*}
with $B_{\rho}(x_l)$ denoting the ball centere at $x_l\in\partial\Omega$ of radius $\rho>0$. The number $\rho>0$ will be determined later. The equation \eqref{eq-main-for-f}, when restricted on the interior domain $\Omega_0$, is nondegenerate. Therefore, the classical Schauder theory for linear parabolic equations implies that 
\begin{equation}\label{eq-interior-C-2-overline-gamma-estimate-of-f}
f\in C_s^{2,\overline{\gamma}}(\Omega_0) \qquad \mbox{at $t=1$}.
\end{equation} 
Hence, we are going to concentrate our attention on the domains $\Omega_l$, $(l\geq 1)$, close to the boundary of $\Omega$. Under the local coordinate change above, we can change the function $f$ defined on $\Omega_l$ into the function $h$ defined on $\mathcal{D}_l\in H$. In addition,
\begin{equation*}
h\in C_s^{2,\overline{\gamma}}(\mathcal{D}_l) \qquad \left(0\leq t<1\right).
\end{equation*}
By the maximality of $\tau_0=1$, $h$ is no longer in $C_s^{2,\overline{\gamma}}(\mathcal{D}_l)$ at time $t=1$. However the existence theory gave $h\in W^{1,2}(\mathcal{D}_l)$ (See chapter 5 in \cite{Va1} for the details). By the Lemmas \ref{L-infty-estimate}, \ref{lem-Gradient-estimate} and \ref{lem-Non-degeneracy-estimate}, the coefficients of $L$ in \eqref{eq-for-sol-h-of-fixed-bondary-problem} and \eqref{eq-aligned-for-h-x-i} satisfy the hypotheses in the corollary \ref{cor-for-holder-estimate-for-h-z} and the Lemma \ref{theorem-Holder-estimates-for-h}, respectively. Then, the Lemma \ref{theorem-Holder-estimates-for-h} and the corollary \ref{cor-for-holder-estimate-for-h-z} tell us that
\begin{equation*}
\nabla h\in C^{\overline{\gamma}}_s(Q^+_{r}(x_l,1)) \qquad \mbox{where $\,\,\overline{\gamma}=\frac{2\gamma}{2-\alpha}$ and $r<1$.}
\end{equation*}
Hence, the coefficients of the equation \eqref{eq-for-sol-h-of-fixed-bondary-problem} satisfy the assumption of Lemma 3.2 in \cite{KL}. By Lemma 3.2 in \cite{KL},
\begin{equation}\label{eq-c-s-1-overline-gamma-estimate-of-h}
h\in C^{2,\overline{\gamma}}_s(Q^+_{\frac{r}{2}}(x_l,1)).
\end{equation}
Let's choose $\rho>0$ so small that $\mathcal{D}_l\in B^+_{\frac{1}{2}}(x_l)$. Then
\begin{equation*} 
h\in C^{2,\overline{\gamma}}_s(\mathcal{D}_l)\qquad \mbox{at $t=1$}.
\end{equation*}
This immediately implies
\begin{equation}\label{eq-boundary-C2-overline-gamma-estimate-of-f} 
f\in C^{2,\overline{\gamma}}_s(\Omega_l)\qquad \mbox{at $t=1$}.
\end{equation}
Therefore, by \eqref{eq-union-of-Omega-by-finite-compact-sets}, \eqref{eq-interior-C-2-overline-gamma-estimate-of-f} and \eqref{eq-boundary-C2-overline-gamma-estimate-of-f}, we get
\begin{equation*}
f\in C^{2,\overline{\gamma}}_s(\Omega) \qquad \mbox{at $t=1$}
\end{equation*}
and theorem follows.
\end{proof}
{\bf Acknowledgement} Sunghoon Kim was supported by Basic Science Research Program through the National Research Foundation of Korea(NRF) funded by the Ministry of Education, Science and Technology(2011-0030749).

\end{document}